\documentclass[a4paper,11pt]{article}

\usepackage[utf8]{inputenc}
\usepackage[T1]{fontenc}
\usepackage{graphicx}
\usepackage{amsmath,amsfonts,amssymb,amsthm}
\usepackage{mathtools}
\usepackage{mathrsfs}
\usepackage{xcolor}
\usepackage{hyperref}
\usepackage{booktabs}
\usepackage{bbold}
\usepackage{bm}
\usepackage{color}
\usepackage[font=footnotesize]{caption}
\usepackage{enumitem}
\usepackage{empheq}
\usepackage{framed}
\usepackage{fullpage}
\usepackage{hyperref}
\usepackage{soul}
\usepackage{dirtytalk}

\mathtoolsset{showonlyrefs}

\definecolor{myred}{HTML}{880000}
\definecolor{mygreen}{HTML}{008800}
\definecolor{myblue}{HTML}{000088}
\definecolor{linkblue}{HTML}{0000BB}

\hypersetup{
  colorlinks,
  linkcolor={myred},
  citecolor={myblue},
  urlcolor={linkblue}
}

\newcommand{\E}{{\mathbf E}}

\renewcommand{\le}{\leqslant}
\renewcommand{\ge}{\geqslant}

\newcommand{\argmin}{\mathop{\mathrm{arg}\,\mathrm{min}}}

\DeclareMathOperator{\tr}{Tr}

\newcommand{\ind}[1]{\bm 1 \left( #1 \right)}

\newcommand{\eps}{\varepsilon}

\renewcommand{\top}{\mathsf{T}}

\usepackage{xspace}

\newcommand{\eg}{e.g.\@\xspace}

\newtheorem{proposition}{Proposition}
\newtheorem{theorem}{Theorem}
\newtheorem*{theorem*}{Theorem}
\newtheorem{lemma}{Lemma}

\theoremstyle{definition}

\theoremstyle{remark}
\newtheorem{remark}{Remark}

\title{Statistically Optimal Robust Mean and Covariance Estimation for Anisotropic Gaussians}
\author{Arshak Minasyan\thanks{CREST, ENSAE, Institut Polytechnique de Paris, \href{mailto:arshak.minasyan@ensae.fr}{arshak.minasyan@ensae.fr}}\quad and\quad Nikita Zhivotovskiy\thanks{University of California, Berkeley. Department of Statistics, \href{mailto:zhivotovskiy@berkeley.edu}{zhivotovskiy@berkeley.edu} }}
\begin{document}

\maketitle
\begin{abstract}
Assume that $X_{1}, \ldots, X_{N}$ is an $\eps$-contaminated sample of $N$ independent Gaussian  vectors in $\mathbb{R}^d$ with mean $\mu$ and covariance $\Sigma$. In the strong $\eps$-contamination model we assume that the adversary replaced an $\eps$ fraction of vectors in the original Gaussian sample by any other vectors. We show that there is an estimator $\widehat \mu$ of the mean satisfying, with probability at least $1 - \delta$, a bound of the form
\[
\|\widehat{\mu} - \mu\|_2 \le c\left(\sqrt{\frac{\tr(\Sigma)}{N}} + \sqrt{\frac{\|\Sigma\|\log(1/\delta)}{N}} + \eps\sqrt{\|\Sigma\|}\right),
\]
where $c > 0$ is an absolute constant and $\|\Sigma\|$ denotes the operator norm of $\Sigma$. In the same contaminated Gaussian setup, we construct an estimator $\widehat \Sigma$ of the covariance matrix $\Sigma$ that satisfies, with probability at least $1 - \delta$,
\[
\left\|\widehat{\Sigma} - \Sigma\right\| \le c\left(\sqrt{\frac{\|\Sigma\|\tr(\Sigma)}{N}} + \|\Sigma\|\sqrt{\frac{\log(1/\delta)}{N}} + \eps\|\Sigma\|\right).
\]
Both results are optimal up to multiplicative constant factors. Despite the recent significant interest in robust statistics, achieving both dimension-free bounds in the canonical Gaussian case remained open. In fact, several previously known results were either dimension-dependent and required $\Sigma$ to be close to identity, or had a sub-optimal dependence on the contamination level $\eps$. 

As a part of the analysis, we derive sharp concentration inequalities for central order statistics of Gaussian, folded normal, and chi-squared distributions. 
\end{abstract}

\section{Robust multivariate mean estimation}
Arguably the first rigorously studied question in robust statistics is the mean (or the location parameter) estimation for contaminated Gaussian distributions \cite{Huber_1964}. A natural extension of this question is a problem of multivariate Gaussian mean estimation when the data is contaminated by a malicious adversary.
When working with uncontaminated data, the celebrated Borell, Tsirelson-Ibragimov-Sudakov \cite{borell1975brunn, cirel1976norms} Gaussian concentration inequality implies the sharp non-asymptotic bound for the performance of the sample mean: If $X_1 \ldots, X_{N}$ are independent Gaussian random vectors in $\mathbb{R}^d$ with mean $\mu$ and covariance $\Sigma$, then, with probability at least $1 - \delta$, 
\begin{equation}
\label{eq:borelsudakovtsirelson}
  \Bigg\|\frac{1}{N}\sum\limits_{i = 1}^N X_i - \mu\Bigg\|_2 \le \sqrt{\frac{\tr(\Sigma)}{N}} + \sqrt{\frac{2\|\Sigma\|\log(1/\delta)}{N}}.  
\end{equation}
Our question is to estimate the mean $\mu$ of a Gaussian random vector when an $\eps$-fraction of all observations is corrupted by a malicious adversary, who knows both the \say{clean} sample and our estimator. We focus on the multivariate case and make no assumptions on corrupted observations. This model is usually called the strong contamination model \cite{diakonikolas2019recent, diakonikolas2020outlier} (for the exact definition of our model see \cite[Definition 1]{dalalyan2022all}) and includes many known contamination models such as Huber's $\eps$-contamination model \cite{Huber_1964}. Such a contaminated sample of Gaussian vectors will be referred to as the $\eps$-\emph{contaminated sample}. Clearly, the sample mean can be compromised even if there is a single outlier, so we will be aiming to provide an analog of \eqref{eq:borelsudakovtsirelson} for a different estimator. Despite significant recent progress in robust statistics (we refer to the recent surveys where both the statistical \cite{lugosi2019mean} and algorithmic \cite{diakonikolas2019recent} aspects are discussed in detail), there is still no sharp analog of the Gaussian concentration inequality \eqref{eq:borelsudakovtsirelson} when the $\eps$-strong adversarial contamination is allowed. Although the Gaussian case is historically the starting point in the theory of robust statistics, our question remains open even from an information-theoretic point of view, without considering computational aspects. Before stating our first bound, we need an additional definition. Given a covariance matrix $\Sigma$, its {\it{effective rank}} is defined as 
\[
    \mathbf{r}(\Sigma) = \frac{\tr(\Sigma)}{\|\Sigma\|},
\]
where $\tr(\Sigma)$ is the trace of matrix $\Sigma$. Obviously, $1 \le \mathbf{r}(\Sigma) \le d$ for a $d$ by $d$ covariance matrix $\Sigma$, but can be much smaller if the distribution of the data is anisotropic and is defined by several principal directions. We are now ready to present our first bound. 
\begin{theorem}[Robust mean estimation in the Gaussian case]
\label{thm:maintheorem}
There are absolute constants $c_1, c_2> 0$ such that the following holds. Assume that $X_{1}, \ldots, X_{N}$ is an $\eps$-contaminated sample of Gaussian random vectors in $\mathbb{R}^d$ with mean $\mu$ and covariance $\Sigma$. Let $\eps < c_1$, then there is an {estimator $\widehat{\mu}$ satisfying}, with probability at least $1 - \delta$,
\[
\|\widehat{\mu} - \mu\|_2 \le c_2\sqrt{\|\Sigma\|}\left(\sqrt{\frac{\mathbf{r}(\Sigma)}{N}} + \sqrt{\frac{\log(1/\delta)}{N}} + \eps\right).
\]
Up to multiplicative constant factors, no estimator of the Gaussian mean can perform better.
\end{theorem}

Although the generality of our result is its main strength, we focus for a moment on the \emph{isotropic case}, that is $\Sigma = I_d$. In this case, if the adversary corrupts at most $O(\sqrt{dN})$ elements of the sample, we still get the optimal performance \eqref{eq:borelsudakovtsirelson} of the sample mean in the setup where the data is not contaminated. This dependence stands in contrast with the more usual $\sqrt{\eps}$-dependence on the contamination level that is known to be achievable under the two moments assumption \cite{lugosi2021robust, diakonikolas2020outlier}. In particular, the latter results only allow $O(d)$ outliers to maintain the optimal performance.

Another interesting aspect of our analysis is that we do not make any assumptions on the sample size. In comparison, the existing estimators that recover the optimal bound in the isotropic case \cite{chen2018robust, diakonikolas2019recent} require $N \ge c(d + \log(1/\delta))$, or $N \ge cd\eps^{-2}$ as in \cite{depersin2021robustness}, where $c > 0$ is some absolute constant. We will encounter some weaker assumptions in Section \ref{sec:tuning}, but only when tuning a single real-valued parameter of our estimator. 

In the context of mean estimation of anisotropic sub-Gaussian distributions, the sharpest known bound is due to Lugosi and Mendelson \cite{lugosi2021robust}. These authors proposed a multivariate version of a trimmed mean estimator that achieves the following error rate 
\begin{equation}
\label{eq:extralog}
\|\widehat{\mu}_{\textrm{LM}} - \mu\|_2 \le c \sqrt{\|\Sigma\|}\left(\sqrt{\frac{\mathbf{r}(\Sigma)}{N}} + \sqrt{\frac{\log(1/\delta)}{N}} + \eps\sqrt{\log\left(\frac{1}{\eps} \right)}\right).
\end{equation} 
The same rate has also been provided by Dalalyan and the first author of this paper \cite{dalalyan2022all} for a different estimator under an additional assumption that $\Sigma$ is known, though using a computationally efficient algorithm.  One may think that the presence of an additional $\sqrt{\log(1/\eps)}$ term is an artifact of the analysis in \cite{lugosi2021robust} and \cite{dalalyan2022all}. This is in fact not true, and the presence of this term is known to be necessary for sub-Gaussian distributions \cite{cheng2019high, lugosi2021robust}. By trimming the observations as in \cite{lugosi2021robust} one can lose some of the specific properties of the Gaussian distribution. Thus, our result asks for both a different estimator and a different analysis. The presence of an additional $\sqrt{\log(1/\eps)}$ term is also interesting from the computational perspective. In particular, the authors of \cite{diakonikolas2017statistical} argue that for any polynomial time Statistical Query algorithm the factor $\sqrt{\log(1/\eps)}$ is unavoidable in the error bound (see also \cite{hopkins2019hard}). We note that in the rich literature on robust statistics, a number of polynomial running time estimators were proposed with $\eps \sqrt{\log(1/\eps)}$ dependence on the contamination level \cite{LaiRV16, DiakonikolasKK017, diakonikolas2019recent, dalalyan2022all, Bateni2022NearlyMR}. 

It is worth mentioning that there are several results showing that the linear dependence on $\eps$ can be achieved in the isotropic case, where $\Sigma$ is identity (or close to it). We refer to the analysis of the Tukey median, the direction-dependent median as well as the Stahel-Donoho median of means in respectively \cite{chen2018robust, diakonikolas2019recent, depersin2021robustness}. Apart from the fact that the isotropic assumption is quite restrictive when working with real data, the presence of VC-type or sphere covering arguments is inherent to the analysis of existing dimension-dependent estimators. Unfortunately, these arguments cannot help us to prove the dimension-free bound of Theorem \ref{thm:maintheorem}. 
It is well understood that the bound on the Gaussian complexity of ellipsoids, which corresponds to the $\sqrt{\tr(\Sigma)}$ term in our result, follows neither from the ellipsoid covering, nor from the Dudley integral, nor from VC-type arguments. This feature is especially pronounced in the Gaussian covariance estimation problem discussed in Section \ref{sec:covariance}. We refer to Chapter 2.5 in the monograph of Talagrand \cite{Talagrand2014} for a thorough analysis of these questions. 

The starting point of our analysis is the folklore property of the sample median of the Gaussian distribution in the one-dimensional case. We denote the sample median by $\operatorname{Med}(\cdot)$ in what follows. If $X_1, \ldots, X_{N}$ is an $\eps$-contaminated sample of independent standard Gaussians with mean $\mu$ and variance $\sigma^2$, then, with probability at least $1 - \delta$, 
\[
\left|\operatorname{Med}(X_1, \ldots, X_{N}) - \mu\right| \le c\sigma\left(\sqrt{\frac{\log(1/\delta)}{N}} + \eps\right),
\]
whenever $N \ge c\log(1/\delta)$ and $\eps$ is smaller than some absolute constant. When going to higher dimensions, instead of working with Tukey's median, whose sharp analysis is only known in the isotropic case \cite{chen2018robust}, or Stahel-Donoho-type estimators as in \cite{depersin2021robustness}, we base our solution on what we call the \emph{smoothed median estimator}. 

\begin{framed}
Let $x_1, \ldots, x_N$ be a set of vectors in $\mathbb{R}^d$, and let $\xi = (\xi_1, \ldots, \xi_N)$ be a zero mean random vector in $\mathbb{R}^N$ whose covariance matrix $H$ is given by $H_{i, j} = \beta^{-1}\langle x_i, x_j\rangle$ for $i, j = 1, \ldots, N$. That is, $H$ is proportional to the Gram matrix of the original data. Here $\beta > 0$ is any integer (chosen by the statistician) satisfying $\mathbf{r}(\Sigma)/10  \le \beta \le 10 \mathbf{r}(\Sigma)$. 
For any direction $v \in S^{d - 1}$, we are interested in the following quantity we call the \emph{smoothed median}:
\begin{equation}
\label{eq:smoothmed}
\operatorname{SmoothMed}_v(x_1, \ldots, x_{N}) = \E_{\xi}\operatorname{Med}(\langle x_1, v\rangle + \xi_1, \ldots, \langle x_{N}, v\rangle + \xi_N).
\end{equation}
Observe that $\operatorname{SmoothMed}_v(x_1, \ldots, x_{N})$ is a function of $x_1, \ldots, x_{N}$ and $v$. The estimator of Theorem \ref{thm:maintheorem} has a simple form. Given an $\eps$-contaminated sample $X_1, \ldots, X_N$, we set
\begin{equation}
    \label{eq:ourestimator}
    \widehat{\mu} = \argmin_{\nu \in \mathbb{R}^{d}}\sup\limits_{v \in S^{d - 1}}\left|\operatorname{SmoothMed}_v(X_1, \ldots, X_{N}) - \langle \nu, v\rangle\right|.
\end{equation}
\end{framed}

\begin{remark}
In Section \ref{sec:tuning}, we discuss how one may choose an integer $\beta$ satisfying $\mathbf{r}(\Sigma)/10  \le \beta \le 10 \mathbf{r}(\Sigma)$ based only on the contaminated sample. Importantly, we avoid a sample-splitting approach when tuning this parameter.
\end{remark}
From the practical perspective, our estimator has complexity exponential in dimension. This is a typical limitation for all existing estimators that have a linear dependence on the contamination level $\eps$ in the Gaussian case\footnote{Recall that existing estimators of this kind lead to dimension-dependent bounds.}. However, in the case when the dimension $d$ is small enough, one can replace the computation over the sphere $S^{d - 1}$ by an appropriate $\varepsilon$-net and approximate the smoothing integration uniformly over all the elements of this net using a Monte Carlo sampling technique. 

The appearance of the smoothed median follows from the proof technique that guarantees a dimension-free nature of our bound. Our approach uses the so-called \emph{PAC-Bayesian lemma}, whose applications were pioneered by O. Catoni and co-authors \cite{audibert2011robust, catoni2016pac, catoni2017dimension} in the context of mean/covariance estimation/linear regression in the heavy-tailed setup. Our application further develops these techniques but in the context of adversarial contamination. An additional discussion appears in Section \ref{sec:techresults}.
\paragraph{Notation.} Throughout the text $c, c_1, c_2, \ldots$ denote absolute constants that may change from line to line. For two positive semi-definite matrices $A$ and $B$ we write $A \preceq B$, if $B - A$ is positive semi-definite. The symbol $\|\cdot\|$ denotes the operator norm of a matrix or the Euclidean norm of a vector depending on the context. Let $\mathbb{S}_{+}^d$ denote the set of $d\times d$ positive semi-definite matrices. The symbol $I_d$ denotes the identity $d \times d $ matrix. We denote the indicator of the event $A$ by $\ind{A}$. For any integer $N$, $[N]$ is the shortened notation of the set $\{1, \ldots, N\}$. For a random variable $Y$ and $\alpha \in [1, 2]$, its $\psi_{\alpha}$ Orlicz norm is defined as follows
\[
\|Y\|_{\psi_{\alpha}} = \inf\{c > 0: \E\exp(|Y|^\alpha/c^{\alpha}) \le 2\}.
\]
Using the standard convention, we say that $\|\cdot\|_{\psi_{2}}$ is the sub-Gaussian norm and $\|\cdot\|_{\psi_1}$ is the sub-exponential norm.   Let $ \mathcal{KL}(\rho, \gamma) = \int\log\left(\frac{d\rho}{d\gamma}\right)d\rho$
denote the Kullback-Leibler divergence between the two measures $\rho$ and $\gamma$ such that $\rho \ll \gamma$. The notation $\rho \ll \gamma$ means that the measure $\rho$ is absolutely continuous with respect to the measure $\gamma$. 
We define the order statistics. Given a set of real numbers $x_1, \ldots, x_n$, let $x_{(1)}, \ldots, x_{(n)}$ denote their non-decreasing rearrangement. That is, $x_{(1)} \le x_{(2)} \le \ldots \le x_{(n)}$. For $\alpha \in [0, 1]$, assuming that $\alpha n$ is an integer, we set
\[
\operatorname{Quant}_{\alpha}(x_1, \ldots, x_n) = x_{(\alpha n)}.
\]
In particular, assuming for simplicity that $n$ is odd, the sample median is given by
\[
\operatorname{Med}(x_1, \ldots, x_n) = x_{((n + 1)/2)}.
\]

\paragraph{Related literature.} Robust statistics is a well-developed topic with several explanatory texts published recently. In our context, the most relevant are the surveys of Lugosi and Mendelson \cite{lugosi2019mean}, and of Diakonikolas and Kane \cite{diakonikolas2019recent}. Some classical references on robust statistics include the textbooks \cite{hampel1980robust,huber1981robust, rousseeuw1987robust}.  When discussing covariance estimation, we refer to the survey \cite{ke2019user} where the focus is on heavy-tailed distributions. We also mention several recent papers on covariance estimation \cite{chen2018robust, minsker2022robust, abdalla2022covariance, oliveira2022improved}, where the focus is on adversarial contamination.

Instead of working with the Euclidean norm as in Theorem \ref{thm:maintheorem}, some authors focus on the Mahalanobis norm. That is, one aims to construct an estimator $\widehat{\mu}$ such that $(\widehat{\mu} - \mu)^{\top}\Sigma^{-1}(\widehat{\mu} - \mu)$ is small with high probability. It appears that the bounds with respect to this norm are necessarily dimension-dependent, and a simple VC-type/sphere covering argument is sufficient to obtain the optimal rates of convergence \cite{depersin2021robustness}. Similar observations are also valid for the covariance estimation problem. We focus on the operator norm, where the analysis allows for dimension-free bounds. This is not the case for the (weighted) Frobenius norm commonly analyzed in the literature.

\section{Covariance estimation}
\label{sec:covariance}
We now move to a more challenging problem of covariance estimation. For simplicity, we assume that our uncontaminated distribution is zero mean. We first need to present a sharp analog of inequality \eqref{eq:borelsudakovtsirelson} in the case where no contamination is allowed. Such a result has been shown only recently by Koltchinskii and Lounici \cite{koltchinskii2017operators}. Their analysis is based on the generic chaining for quadratic processes. This non-trivial approach is motivated by the difficulty of replacing $d$ with the effective rank $\mathbf{r}(\Sigma)$. Let us formulate their result. Assume that $Y_1, \ldots, Y_N$ are independent zero mean Gaussian vectors in $\mathbb{R}^d$ with covariance $\Sigma$. There are absolute constants $c_1, c_2 > 0$ such that, with probability at least $1 - \delta$,
\begin{equation}
\label{eq:samplecov}
\left\|\frac{1}{N}\sum\limits_{i = 1}^NY_iY_i^{\top} - \Sigma\right\| \le c_1\|\Sigma\|\left(\sqrt{\frac{\mathbf{r}(\Sigma)}{N}} + \sqrt{\frac{\log(1/\delta)}{N}}\right),
\end{equation}
provided that $N \ge c_2(\mathbf{r}(\Sigma) + \log(1/\delta))$. When adversarial contamination is allowed, the sharpest known dimension-free result is implied by the bound of Abdalla and the second author of this paper \cite{abdalla2022covariance}. Their work suggests a trimmed-mean-based estimator that achieves the rate
\begin{equation}
\label{eq:azh}
\left\|\widehat{\Sigma}_{\textrm{AZh}} - \Sigma\right\| \le c_1\|\Sigma\|\left(\sqrt{\frac{\mathbf{r}(\Sigma)}{N}} + \sqrt{\frac{\log(1/\delta)}{N}} + \eps\log\left(\frac{1}{\eps}\right)\right),
\end{equation}
whenever $N \ge c_2(\mathbf{r}(\Sigma) + \log(1/\delta))$. 
The above bound is valid for any sub-Gaussian distribution and cannot be improved in general. However, similarly to the case of mean estimation, we expect a better dependence on the contamination parameter $\eps$ in the Gaussian case. We also note that simpler versions of the bound \eqref{eq:azh} based on the median-of-means estimators \cite{mendelson2020robust} only lead to a $\sqrt{\eps}$-dependence on the contamination level.
\begin{remark}
As a side note, when aiming for a sharp leading constant in \eqref{eq:samplecov} in the uncontaminated setup, an almost optimal performance follows from the recent result of Han \cite{han2022exact} combined with the second order concentration inequality derived in \cite{klochkov2020uniform}. A similar bound in the sub-Gaussian case with explicit constants can be found in \cite{zhivotovskiy2021dimension}.
\end{remark}

In our analysis, we first make some additional assumptions. Our estimator depends on some additional parameters that could be pre-estimated based only on the observed $\varepsilon$-contaminated sample. A careful analysis of pre-estimation procedures is deferred to Section \ref{sec:tuning}. For the rest of this section, we assume that we have an access to the following quantities.
\begin{enumerate}
\item There is an integer $\beta$ and a real number $\omega$ satisfying respectively 
\[
 \frac{1}{10}\mathbf{r}(\Sigma) \le \beta \le 10\mathbf{r}(\Sigma), \quad \text{and}\quad \frac{1}{10}\|\Sigma\| \le \omega \le 10\|\Sigma\|.
\]
\item Let $H$ be a known positive semi-definite matrix. Assume that we know a real number $\alpha = \alpha(H)$ satisfying
\begin{equation}
\label{eq:alphah}
|\alpha - \tr(\Sigma H)| \le c\tr(\Sigma H)\left(\sqrt{\frac{\mathbf{r}(\Sigma) + \log(1/\delta)}{N}} + \eps\right),
\end{equation}
for some absolute constant $c > 0$. 
\item We have an access to a positive semi-definite matrix $G$ satisfying 
\begin{equation}
\label{eq:gmatrix}
\frac{1}{10}\Sigma \preceq G, \quad \textrm{and}\quad \tr(G) \le 10\tr(\Sigma).
\end{equation}
\end{enumerate}
Except for the matrix $G$, we only need to tune real valued parameters.  This can be usually done under the minimal assumption $N \ge c(\mathbf{r}(\Sigma) + \log(1/\delta))$ for some absolute constant $c > 0$. Observe that the assumption $\Sigma \preceq 10G$ does not imply that $\Sigma$ is close to $G$ in the operator norm. At the same time, this assumption requires some control over the smallest singular value of $\Sigma$. We show, in particular, that whenever $N \ge c(d + \log(1/\delta))$, we can always efficiently construct such a matrix $G$ based on contaminated data, while still maintaining the dimension-free nature of our upper bound. Moreover, it appears that our guarantees are uniform with respect to the choice of the matrix $G$. One can rerun our estimator on the same data multiple times with any admissible $G$ satisfying \eqref{eq:gmatrix} without affecting the performance of our estimator. We discuss this formally in Section \ref{sec:tuning}. 
\begin{theorem}[Robust covariance estimation in the Gaussian case]
\label{thm:covarianceestimation}
There are absolute constants $c_1, c_2 > 0$ such that the following holds. Assume that $X_{1}, \ldots, X_{N}$ is an $\eps$-contaminated sample of zero mean  Gaussian vectors in $\mathbb{R}^d$ with covariance $\Sigma$. Let $\eps < c_1$, then there is an estimator $\widehat{\Sigma} = \widehat{\Sigma}_{\alpha, \beta, \omega, G, \eps}(X_1, \ldots, X_N)$ satisfying, with probability at least $1 - \delta$,
\[
\left\|\widehat{\Sigma} - \Sigma\right\| \le c_2\|\Sigma\|\left(\sqrt{\frac{\mathbf{r}(\Sigma)}{N}} + \sqrt{\frac{\log(1/\delta)}{N}} + \eps\right).
\]
Up to multiplicative constant factors, no estimator of the covariance matrix performs better. 
\end{theorem}
We are now ready to define our estimator. 
\begin{framed}
We first construct the following distribution. For any $v \in S^{d - 1}$, let $\rho_{v}$ be a distribution in $\mathbb{R}^d$ whose density $f_{v}$ is given by
\[
f_{v}(x) = \frac{1}{p(2\pi\beta^{-1})^{d/2}}\exp\left(-\frac{\beta\|x - v\|^2}{2}\right)\ind{\|G^{1/2}(x - v)\| \le 100\sqrt{\omega}}.
\]
Here $p > 0$ is a normalization factor.  Assume that $\theta$ is a random vector distributed according to $\rho_{v}$. Let $H = \E_{\rho_v}(\theta - v)(\theta - v)^{\top}$ be a covariance matrix of $\rho_{v}$ (by the symmetry of $\rho_v$ around $v$, the matrix $H$ does not depend on $v$). For some specifically chosen absolute constant $c > 0$ and $\alpha = \alpha(H)$, define the set 
\begin{equation}
\label{eq:setH}
\mathcal{H} = \Bigg\{\Gamma\in \mathbb{S}_{+}^d: |\tr(\Gamma H) -\alpha| \le c\tr(\Gamma H)\Bigg(\sqrt{\frac{\mathbf{r}(G) + \log({1}/{\delta})}{N}} + \eps\Bigg);\Gamma \preceq 10G; \|\Gamma\| \le 10\omega \Bigg\}.
\end{equation}
For an $\eps$-contaminated sample $X_1, \dots, X_N$, our estimator is defined as follows:
\begin{equation}
    \label{eq:Cov-estimator}
    \widehat{\Sigma} = \argmin_{\Gamma \in \mathcal{H}}\sup\limits_{v \in S^{d - 1}}\E_{\rho_{v}}\left|\operatorname{Med}\left(|\langle X_1, \theta\rangle|, \ldots, |\langle X_N, \theta\rangle|\right) - \Phi^{-1}(3/4)\sqrt{\theta^{\top}\Gamma\theta}\right|.
\end{equation}
\end{framed}

This estimator is a more complex version of our smoothed median estimator. 
First, instead of working with the Gaussian smoothing measure, we restrict this distribution to an elliptic set $\{x \in \mathbb{R}^d: \|G^{1/2}(x - v)\| \le 100\sqrt{\omega}\}$. Second, we need to restrict the eigenvalues of the output matrix and introduce the set $\mathcal{H}$. Our estimator is related to minimizing the so-called \emph{median absolute deviation} (see \cite{donoho1992breakdown} for related definitions). The proof of Theorem \ref{thm:covarianceestimation} exploits the fact that quantiles of $|\langle X, \theta\rangle|$ are tightly connected with corresponding variances. This is reflected in the term $\Phi^{-1}(3/4)$ appearing in the definition of our estimator.

\section{Concentration inequalities for sample quantiles}

In this section, we obtain sub-Gaussian and sub-exponential concentration inequalities for quantiles of i.i.d. observations sampled according to several regular distributions. The analysis of sample quantiles is a standard question in statistics. The early work of Kolmogorov \cite{kolmogorov31} focused on proving a central limit theorem for the sample median of some symmetric distributions. Subsequently, the focus was on explicit expansions for this limit law \cite{burnashev1997asymptotic}. Another line of research focused on studying the explicit formulas for the distribution of order statistics. We refer to the monograph of David and Nagaraja \cite{david2004order} on order statistics and to the monograph of De Haan and Ferreira \cite{de2007extreme} on the extreme value theory, a topic covering the properties of smallest and largest values in samples. Finally, many authors focused on the analysis of Bahadur's representation of sample quantiles (see e.g., \cite{bahadur1966note, kiefer1967bahadur}, \cite[Theorem 5.11]{shao2003mathematical}). Unfortunately, neither the exact expressions for the distribution of sample quantiles nor various asymptotic expansions lead to the exact concentration inequalities we are interested in. 

Less is known about concentration inequalities for sample quantiles. Some explicit non-asymptotic bounds appear in the monograph of Shao \cite[Section 5.3]{shao2003mathematical}, where a reduction to a concentration of Bernoulli random variables is made. Several related concentration inequalities appear in the work of Boucheron and Thomas \cite{boucheron2012concentration}, though their bounds are not sharp enough for our purposes. In particular, these authors provide a sub-exponential concentration inequality for the sample median of the Gaussian distribution, while our results will lead to sub-Gaussian concentration inequalities. More recent results on sample quantiles are inspired by the problems in robust statistics. In fact, the analysis in \cite{chen2018robust, altschuler2019best} provides some sharp bounds for sample quantiles, though the existing bounds do imply the sub-Gaussian concentration only for small enough deviations from the mean. Another line of results is due to Bobkov and Ledoux \cite{bobkov2019one}. Their results provide sharp concentration inequalities for log-concave distributions (recall that the Gaussian distribution is log-concave), but only lead to sub-exponential tails due to their generality\footnote{For the special case of the uniform distribution on the real line Bobkov and Ledoux \cite{bobkov2019one} provide a sub-Gaussian concentration inequality for all order statistics of the uniform distribution in $[0, 1]$.}.

Our approach is quite simple, though, to the best of our knowledge, it is not used explicitly in the literature. When proving concentration inequalities for sample quantiles, we consider two regimes. For small deviations, we use the regularity of the density function and follow the reduction to a concentration of Bernoulli random variables as in \cite{shao2003mathematical, chen2018robust, altschuler2019best, diakonikolas2019recent, xia2019non}, while for large deviations we use the sub-Gaussian/sub-exponential tails of our distribution. This leads to desired sharp concentration inequalities. We discuss some straightforward extensions of our analysis in Section \ref{sec:conlcudingremarks}.

Before providing our first concentration inequality, recall that the cumulative distribution function of a standard Gaussian is denoted by $\Phi(\cdot)$. Denote its inverse by $\Phi^{-1}(\cdot)$. 

\begin{lemma}[Concentration for Gaussian quantiles]
\label{lem:subGquantiles}
There are absolute constants $c_1, c_2 > 0$ such that the following holds. Let $\eps \in [0, 1/4]$. Assume without loss of generality that $(1/2 \pm \eps)N$ are integers. Let $Y_1, \ldots, Y_N$ be a sample of independent standard Gaussian random variables.
Then, for any $t \ge 0$,
    \[
       \Pr(|Y_{((1/2 \pm \eps)N)} - \Phi^{-1}(1/2 \pm \eps)|\ge t) \le 2\exp(-c_1 Nt^2).
    \]
Equivalently, 
\[
\left\|Y_{((1/2 \pm \eps)N)} - \Phi^{-1}(1/2 \pm \eps)\right\|_{\psi_2} \le \frac{c_2}{\sqrt{N}}.
\]
\end{lemma}

\begin{proof}
We only analyze the quantile $Y_{((1/2 + \eps)N)}$, since the analysis for $Y_{((1/2 - \eps)N)}$ is the same. We analyze two parts of the tail separately. First, we show that 
\[
\Pr(Y_{((1/2 + \eps)N)} - \Phi^{-1}(1/2 + \eps)\ge t) \le \exp(-c_1 Nt^2).
\]
This analysis is also split into two regimes. For some absolute constant $C > 0$, we first show the above inequality for $0 \le t \le C$ and then proceed with the case $t \ge C$. In the first regime, we follow the standard reduction to binomial tails (see similar computations in \cite[Theorem 5.9]{shao2003mathematical} and \cite{xia2019non}). Let $Z_1, \ldots, Z_N$ be independent Bernoulli random variables with expectation $p = 1- \Phi(s)$ for some fixed $s \in \mathbb{R}$. We have 
\[
\Pr(Y_{((1/2 + \eps)N)} \ge s) = \Pr\left(\sum\nolimits_{i = 1}^N Z_i > (1/2 - \eps)N\right).
\]
We set $s = \Phi^{-1}(1/2 + \eps)+ t$ and obtain, using that $\E Z_i = 1 - \Phi(\Phi^{-1}(1/2 + \eps)+ t)$,
\begin{align*}
\Pr(Y_{((1/2 + \eps)N)} \ge \Phi^{-1}(1/2 + \eps)+ t) &= \Pr\left(\frac{1}{N}\sum\limits_{i = 1}^N Z_i - \E Z_i > \Phi(\Phi^{-1}(1/2 + \eps)+ t) - \frac{1}{2} - \eps\right).
\end{align*}
Denoting $\varphi_+(t) = \Phi(\Phi^{-1}(1/2 + \eps)+ t) - 1/2 - \eps$, we have by Hoeffding's inequality applied to independent Bernoulli random variables, whenever $\varphi_+(t) \ge 0$,
\begin{equation}
\label{eq:uppertailphi-gauss}
\Pr(Y_{((1/2 + \eps)N)} -  \Phi^{-1}(1/2 + \eps)\ge t) \le \exp\left(-2N(\varphi_+(t))^2\right).
\end{equation}
Let us lower bound the function $\varphi_+$. 
Using the density formula, we have
\begin{align*}
    \varphi_+(t) &= \Phi(\Phi^{-1}(1/2 + \eps) + t) - \Phi(\Phi^{-1}(1/2 + \eps))
    \\
    &\ge \frac{t}{\sqrt{2\pi}}\exp\left(-(\Phi^{-1}(1/2 + \eps) + t)^2/2\right)
    \\
    &\ge \frac{t}{\sqrt{2\pi}}\exp\left(-(\Phi^{-1}(3/4) + t)^2/2\right).
\end{align*}
We combine these computations with the tail for large values of $t$. Since $\eps \le 1/4$, we need that at least $N/4$ (assume that it is an integer without loss of generality) of all observations are above $\Phi^{-1}(1/2 + \eps) + t$. This can be controlled as follows
\begin{align*}
\Pr(Y_{((1/2 + \eps)N)} - \Phi^{-1}(1/2 + \eps) \ge t) &\le \binom{N}{N/4}\left(\Pr(Y_1 \ge \Phi^{-1}(1/2 + \eps) + t)\right)^{N/4}
\\
&\le 2^N\left(\Pr(Y_1 \ge  t)\right)^{N/4}
\\
&\le 2^N\exp(-Nt^2/8) 
\\
&= \exp(N\log(2) - Nt^2/8)
\\
&\le \exp(-Nt^2/16),
\end{align*}
whenever $t \ge 4\sqrt{\log(2)}$.
The inequality \eqref{eq:uppertailphi-gauss} and the lower bound on $\varphi_+(t)$ give us
\[
\Pr(Y_{((1/2 + \eps)N)} - \Phi^{-1}(1/2 + \eps) \ge t) \le \exp\left(\frac{-2Nt^2}{2\pi \exp((\Phi^{-1}(3/4)+4\sqrt{\log(2)})^{{2}})}\right),
\]
whenever $0 \le t \le 4\sqrt{\log(2)}$. Combining two regimes and adjusting the absolute constant, we prove an upper tail. Let us prove the lower tail bound. The proof is similar, though the computations are slightly different. We want to show for any $t \ge 0$,
\[
\Pr(Y_{((1/2 + \eps)N)} - \Phi^{-1}(1/2 + \eps) \le - t) \le \exp(-c_1 Nt^2).
\]
We have for any $s \in \mathbb{R}$ and $Z_1, \ldots, Z_N$ as above
\[
\Pr(Y_{((1/2 + \eps)N)} \le s) = \Pr\left(\sum\nolimits_{i = 1}^N Z_i \le (1/2 - \eps)N\right).
\]
Define
\[
\varphi_-(t) = \Phi(\Phi^{-1}(1/2 + \eps) - t) - 1/2 - \eps .
\]
We have similarly 
\[
\Pr(Y_{((1/2 + \eps)N)} - \Phi^{-1}(1/2 + \eps) \le -t) \le \exp\left(-2N(\varphi_-(t))^2\right).
\]
Now we lower bound the quantity $|\varphi_-(t)|$ as follows  
\begin{align*}
    |\varphi_-(t)| &= \Phi(\Phi^{-1}(1/2 + \eps)) - \Phi(\Phi^{-1}(1/2 + \eps) - t)
    \\
    &\ge \frac{t}{\sqrt{2\pi}}\exp\left(-\max\{(\Phi^{-1}(1/2 + \eps))^2, (\Phi^{-1}(1/2 + \eps) - t)^2\}/2\right)
    \\
    &\ge \frac{t}{\sqrt{2\pi}}\exp\left(-\max\{(\Phi^{-1}(3/4))^2,  t^2\}/2\right).
\end{align*}
As above, we need to get the tail for large values of $t$. We need that at least $N/{4}$ of all observations are below $\Phi^{-1}(1/2 + \eps) - t$. In what follows, we assume $t \ge 2\Phi^{-1}(3/4)$. This can be controlled as follows
\begin{align*}
\Pr(Y_{((1/2 + \eps)N)} - \Phi^{-1}(1/2 + \eps) \le -t) &= \Pr(-Y_{((1/2 + \eps)N)} \ge t - \Phi^{-1}(1/2 + \eps))
\\
&\le \binom{N}{N/{4}}\left(\Pr(-Y_1 \ge t - \Phi^{-1}(1/2 + \eps))\right)^{N/{4}}
\\
&\le  2^N\left(\Pr(Y_1 \ge t - \Phi^{-1}(3/4))\right)^{N/{4}}
\\
&\le 2^N\left(\Pr(Y_1 \ge t/2)\right)^{N/{4}}
\\
&\le 2^N\exp(-Nt^2/{32}) 
\\
&\le \exp(-Nt^2/{64}),
\end{align*}
whenever $t \ge \max\{2\Phi^{-1}(3/4), {8}\sqrt{\log(2)}\}$. The proof of the lower tail follows. The union bound concludes the proof. Finally, our bound on the $\psi_2$-norm follows from \cite[Proposition 2.5.2]{Vershynin2016HDP}.
\end{proof}

Our second result presents a similar concentration bound for the empirical quantiles of i.i.d. observations drawn from $\chi^2_1$ distribution. This distribution coincides with the distribution of the squared standard Gaussian random variable. Denote the cumulative distribution function by $F_{\chi^2_1}(\cdot)$ and its inverse by $F_{\chi^2_1}^{-1}(\cdot)$. The key difference is that we only show a sub-exponential tail in this case. We remark that when considering the $\chi^2_k$ distribution with $k \ge 2$ degrees of freedom, the desired concentration inequality follows from log-concavity and \cite[Lemma 6.5]{bobkov2019one}.   

\begin{lemma}[Quantiles of the $\chi^2_1$ distribution]
\label{lem:orderstatofchisq}
There is an absolute constant $c_1 > 0$ such that the following holds. Assume without loss of generality that $(1/2 \pm \eps)N$ are integers. Let $Y_1, \ldots, Y_N$ be a sample of independent $\chi^2_1$ random variables and $\eps \in [0, 1/4]$.
Then, 
\[
\left\|Y_{((1/2 \pm \eps)N)} - F_{\chi^2_1}^{-1}(1/2 \pm \eps)\right\|_{\psi_1} \le \frac{c_1}{\sqrt{N}}.
\]
\end{lemma}

\begin{proof}
First, the density $g$ of $\chi_1^2$ is given by 
\begin{equation}
    \label{eq:expden}
    g(x) = \frac{\exp(-x/2)}{\sqrt{2\pi x}}, \quad x > 0.
\end{equation}
It is also easy to show that $F^{-1}_{\chi_1^2}(1/2) = (\Phi^{-1}(3/4))^2$.
Using the same notation again, we denote $\varphi_+(t) = F_{\chi_1^2}(F^{-1}_{\chi_1^2}(1/2 + \eps)+ t) - 1/2 - \eps$. Repeating the lines of the proof of Lemma \ref{lem:subGquantiles}, whenever $\varphi_+(t) \ge 0$, we have
\begin{equation}
\label{eq:uppertailphi-chisq}
\Pr(Y_{((1/2 + \eps)N)} -  F^{-1}_{\chi_1^2}(1/2 + \eps)\ge t) \le \exp\left(-2N(\varphi_+(t))^2\right).
\end{equation}
Using \eqref{eq:expden}, we have
\begin{align*}
\varphi_+(t) &= F_{\chi_1^2}(F^{-1}_{\chi_1^2}(1/2 + \eps)+ t) - F_{\chi_1^2}(F^{-1}_{\chi_1^2}(1/2 + \eps))
\\
&\ge \frac{t\exp(-F_{\chi_1^2}(F^{-1}_{\chi_1^2}(1/2 + \eps)+ t))}{\sqrt{2\pi F_{\chi_1^2}(F^{-1}_{\chi_1^2}(1/2 + \eps)+ t)}}.
\end{align*}
This gives us a sub-Gaussian tail for as long as $\eps \le 1/4$ and $t \ge 0$ is bounded by some absolute constant. By the concentration of the $\chi^2$ distribution \cite[Lemma 1]{laurent2000adaptive} we have for $t \ge 0$,
\[
\Pr(Y_1 \ge 1 + \sqrt{2t} + t) \le \exp(-t),\ \textrm{and thus},\  \Pr(Y_1 \ge t) \le \exp(-t/2), \ \textrm{whenever}\ t \ge 4 + 2\sqrt{3}.
\]
Since $\eps \le 1/4$, we need that at least $N/4$ of all observations are above $F^{-1}_{\chi_1^2}(1/2 + \eps) + t$. Therefore, whenever $t \ge 4 + 2\sqrt{3}$, we have
\begin{align*}
\Pr(Y_{((1/2 + \eps)N)} - F^{-1}_{\chi_1^2}(1/2 + \eps) \ge t) &\le \binom{N}{N/4}\left(\Pr(Y_1 \ge \Phi^{-1}(1/2 + \eps) + t)\right)^{N/4}
\\
&\le\binom{N}{N/4}\left(\Pr(Y_1 \ge t)\right)^{N/4}
\\
&\le 2^N\exp(-Nt/8)
\\
&= \exp(N\log(2) - Nt/8)
\\
&\le \exp(-Nt/16),
\end{align*}
where the last inequality requires additionally $t \ge 16\log (2)$. Combining the above bounds and adjusting the absolute constant $c_1 \ge 0$ we show that 
\[
\Pr(Y_{((1/2 + \eps)N)} - F^{-1}_{\chi_1^2}(1/2 + \eps) \ge t) \le \exp(-c_1N\min\{t, t^2\}).
\]
We continue with the bound on the lower tail. We want to show for any $t \ge 0$,
\[
\Pr(Y_{((1/2 + \eps)N)} - F^{-1}_{\chi_1^2}(1/2 + \eps) \le - t) \le \exp(-c_2Nt^2),
\]
where $c_2$ is an absolute constant.
For $0 \le t < F^{-1}_{\chi_1^2}(1/2 + \eps)$ we define
\[
\varphi_-(t) = F^{-1}_{\chi_1^2}(F^{-1}_{\chi_1^2}(1/2 + \eps) - t) - 1/2 - \eps .
\]
Using the same argument we show
\[
\Pr\left(Y_{((1/2 + \eps)N)} - F^{-1}_{\chi_1^2}(1/2 + \eps) \le -t\right) \le \exp\left(-2N(\varphi_-(t))^2\right).
\]
We have 
\[
    |\varphi_-(t)| = F_{\chi_1^2}(F^{-1}_{\chi_1^2}(1/2 + \eps)) - F_{\chi_1^2}(F^{-1}_{\chi_1^2}(1/2 + \eps) - t) \ge \frac{t\exp\left(-F^{-1}_{\chi_1^2}(1/2 + \eps)\right)}{\sqrt{2\pi F^{-1}_{\chi_1^2}(1/2 + \eps)}}.
\]
Since $\eps \le 1/4$ we conclude the proof in the regime $t \le F^{-1}_{\chi_1^2}(1/2 + \eps)$.
Observe that due to the non-negativity of $Y_{((1/2 + \eps)N)}$ we can extend this bound to all $t > F^{-1}_{\chi_1^2}(1/2 + \eps)$. Lemma \ref{lem:expmoment} in Section \ref{sec:techresults} concludes the proof. The analysis of the $(1/2-\eps)$-th quantile repeats the same lines.
\end{proof}

Our final result proves a similar bound for the (standard) half-normal distribution. Namely, we want to prove the concentration inequality for quantiles of the absolute values of standard Gaussian random variables. Denote the cumulative distribution function of this distribution by $\Phi_{\operatorname{H}}(\cdot)$ and its inverse by $\Phi^{-1}_{\operatorname{H}}(\cdot)$. 
\begin{lemma}[Quantiles of the half-normal distribution]
\label{lem:orderstatsofhalfnormal}
There is an absolute constant $c_1$ such that the following holds. Assume without loss of generality that $(1/2 \pm \eps)N$ are integers. Let $Y_1, \ldots, Y_N$ be a sample of independent half-normal random variables and $\eps \in [0, 1/4]$.
Then, 
\[
\left\|Y_{((1/2 \pm \eps)N)} - \Phi^{-1}_{\operatorname{H}}(1/2 \pm \eps)\right\|_{\psi_2} \le \frac{c_1}{\sqrt{N}}.
\]
\end{lemma}
The proof of this result repeats the same computations used in the proofs of Lemma \ref{lem:subGquantiles} and Lemma \ref{lem:orderstatofchisq}. We omit the details.

\section{Auxiliary results}
\label{sec:techresults}
The following section contains several technical results used throughout the paper.
We start with a bound usually referred to as the \emph{PAC-Bayesian lemma}, which is a direct consequence of the Donsker–Varadhan’s
variational formula for the relative entropy \cite{donsker1975asymptotic}. 

\begin{lemma}
\label{lem:pacbayes}
Assume that $X$ is a random variable defined on some measurable space $\mathcal X$. Assume also that $\Theta$ (called the parameter space) is a subset of $\mathbb{R}^d$. Let $\gamma$ be a distribution (called prior) on $\Theta$ and let $\rho$ be any distribution (called posterior) on $\Theta$ such that $\rho \ll \gamma$. Let $f: \mathcal X \times \Theta \to \mathbb{R}$ be such that $\E_X\exp(f(X, \theta))$ is finite $\gamma$-almost surely. Then, we have 
\begin{align*}
    \Pr_{X}\Big(\textrm{for all}\  \rho \ll \gamma ~:~  \E_{\rho}f(X, \theta) \le \E_{\rho}\log(\E_X\exp(f(X, \theta))) + \mathcal{KL}(\rho, \gamma) + t \Big) \ge 1 - e^{-t}.
\end{align*}
\end{lemma}

One of the key arguments, used in several recent papers on mean and covariance estimation of heavy-tailed distributions \cite{catoni2017dimension, giulini2018robust, abdalla2022covariance, oliveira2022improved}, is a skillful application of this lemma allowing to bypass the sphere-covering and VC-type arguments. Lemma \ref{lem:pacbayes} will play the same key role in our analysis. However, previous applications of this lemma were based on sums of (truncated) random variables in place of $f(X, \theta)$ (in this case $X$ is essentially a vector of independent random variables $X_1, \ldots, X_N$), while we are exploiting the interplay between Lemma \ref{lem:pacbayes} and sample quantiles of particular univariate distributions. The proof of Lemma \ref{lem:pacbayes} and some of its applications can be found in \cite{catoni2017dimension, zhivotovskiy2021dimension}. 

\subsection{Analysis of the posterior distribution}
Another technical aspect of our analysis is the introduction of truncated posterior distributions in the context of robust estimation. 
For a given positive semi-definite matrix $G$ and $r \ge 0$, we truncate the multivariate Gaussian distribution with mean $v \in S^{d - 1}$ and covariance $\beta^{-1}I_d$ as follows. Define the density function 
\begin{equation}
\label{eq:measureforkl}
f_v(x) = \frac{1}{p(2\pi\beta^{-1})^{d/2}}\exp\left(-\frac{\beta\|x - v\|^2}{2}\right)\ind{\|G^{1/2}(x - v)\| \le r},
\end{equation}
where $p > 0$ is a normalization factor. 
We proceed with the following result. 

\begin{lemma}[Properties of the truncated posterior]
\label{lem:kllemma}
Let $r, \beta > 0$ and let $G$ denote a positive semi-definite in the definition \eqref{eq:measureforkl}. 
Let $\Sigma$ be a covariance matrix of a zero mean random vector $X$ in $\mathbb{R}^d$ satisfying
\[
\frac{1}{10}\Sigma \preceq G, \quad \textrm{and}\quad \tr(G) \le 10\tr(\Sigma).
\]
Let $\gamma$ be a Gaussian measure in $\mathbb{R}^d$ with mean zero and covariance $\beta^{-1}I_d$. If, additionally
\[
r \ge \sqrt{20\beta^{-1}\tr(\Sigma)},
\]
then we have
\[
\mathcal{KL}(\rho_v, \gamma) \le \log(2) + \beta/2.
\]
Furthermore, let $\theta$ be distributed according to $\rho_v$. Then, $\E_{\rho_v}\theta = v$, and almost surely with respect to the realization of $\theta$, we have
\begin{equation}
\label{eq:almostsurebound}
\theta^{\top}\Sigma\theta \le 2\|\Sigma\| + 20r^2.
\end{equation}

\end{lemma}
\begin{proof}
We use that for $\theta$ distributed according to $\rho_v$ it holds that $\E_{\rho_v} \theta = v$. This follows from the symmetry of the density around $v$. Let $g$ denote the density of a Gaussian random vector with mean zero and covariance $\beta^{-1}I_d$. To control $\mathcal{KL}(\rho_v, \mu)$ we write 

\begin{align*}
\int\log\left(\frac{f_v(x)}{g(x)}\right)f_v(x)dx &= \E_{\rho_{v}}\log\left(\frac{1}{p}\exp\left(\frac{-\beta\|\theta - v\|^2 + \beta\|\theta\|^2}{2}\right)\right)
\\
&=\log\left(\frac{1}{p}\right) +  \E_{\rho_{v}}\left(\frac{-\beta\|v\|^2 + 2\beta\langle \theta, v\rangle}{2}\right)
\\
&= \log\left(\frac{1}{p}\right) + \frac{\beta}{2}.
\end{align*}
To prove the desired inequality we observe that
\[
p = \Pr(\|G^{1/2}W\| \le r),
\]
where $W$ is a zero mean Gaussian random vector with covariance $\beta^{-1}I_d$. Since $\tr(G) \le 10\tr(\Sigma)$, a simple computation shows that
\[
\Pr(\|G^{1/2}W\| \ge r) \le \E W^{\top}GW/r^2 = \beta^{-1}\tr(G)/r^{2} \le 10\beta^{-1}\tr(\Sigma)/r^{2}  \le 1/2,
\]
as long as $r \ge \sqrt{20\beta^{-1}\tr(\Sigma)}$. Thus, under this assumption $p \ge 1/2$, and hence $\log(1/p) \le \log 2$. This proves the first inequality. Using the second property of the matrix $G$, we have
\[
\theta^{\top}\Sigma\theta \le 2v^{\top}\Sigma v + 2(\theta - v)^{\top}\Sigma(\theta - v) \le 2v^{\top}\Sigma v + 20(\theta - v)^{\top}G(\theta - v)\le  2\|\Sigma\| + 20r^2.
\]
The claim follows.
\end{proof}

Our next result convertes a mixed sub-Gaussian/sub-exponential tail bound into a bound on the $\|\cdot\|_{\psi_1}$-norm. We present this standard computation for the sake of completeness.
\begin{lemma}
\label{lem:expmoment}
Assume that a random variable $X$ satisfies for all $t \ge 0$,
\[
\Pr(|X| \ge t) \le 2\exp(-K\min\{t^2, t\}),
\]
where $K > 1$ is some constant. Then, there is an absolute constant $c > 0$ such that
\[
\|X\|_{\psi_1} \le \frac{c}{\sqrt{K}}.
\]
\end{lemma}
\begin{proof}
We can simply compute the moments of $X$. For fixed $p \ge 1$, we have
\begin{align*}
\E|X|^p &= \int\limits_{0}^{\infty}\Pr(|X|^p \ge t) ~dt= \int\limits_{0}^{\infty}\Pr(|X| \ge t)pt^{p - 1}~dt \le 2\int\limits_{0}^{\infty}\exp(-K\min\{t^2, t\})pt^{p - 1}dt 
\\
&\le 2\int\limits_{0}^{\infty}\exp(-Kt^2)pt^{p - 1}dt + 2\int\limits_{0}^{\infty}\exp(-Kt)pt^{p - 1}dt = \frac{1}{K^{p/2}}p\Gamma(p/2) + \frac{1}{K^p}2p\Gamma(p)
\\
&\le \frac{3p(p/2)^{p/2}}{K^{p/2}} + \frac{2p^p}{K^p} \le \frac{3p^p}{K^{p/2}} + \frac{2p^p}{K^p} \le \frac{5p^p}{K^{p/2}},
\end{align*}
where $\Gamma(\cdot)$ stands for the gamma function, and we used $\Gamma(x) \le 3x^x$ for all $x \ge 1/2$ together with $p\Gamma(p) = \Gamma(p+1) \le p^p$. Finally, \cite[Proposition 2.7.1, {(b)}]{Vershynin2016HDP} implies the desired bound.
\end{proof}

\section{Proofs of main results}
We begin with the proof of our first main result that yields that the estimator defined in \eqref{eq:ourestimator} achieves an optimal error bound for the robust mean estimation problem. We discuss the optimality of our results at the end of this section.
\begin{proof}[Proof of Theorem \ref{thm:maintheorem}.]
First, by the definition of our estimator, we have
\[
\widehat{\mu} = \argmin_{\nu \in \mathbb{R}^d}\sup\limits_{v \in S^{d - 1}}|\E_{\rho_{v}}\operatorname{Med}(\langle X_1, \theta \rangle, \ldots, \langle X_{N}, \theta \rangle) - \langle \nu, v\rangle|,
\]
where $\rho_v$ is a multivariate Gaussian distribution in $\mathbb{R}^d$ with mean $v$ and covariance $\beta^{-1}I_d$. By the triangle inequality and the definition of our estimator, we have
\begin{align*}
    \|\widehat{\mu} - \mu\| &= \sup_{v \in S^{d-1}} \langle \widehat{\mu} - \mu, v \rangle \\
    &\le \sup\limits_{v \in S^{d - 1}}|\E_{\rho_{v}}\operatorname{Med}(\langle X_1, \theta \rangle, \ldots, \langle X_{N}, \theta \rangle) - \langle \widehat{\mu}, v\rangle|
    \\
    &\quad+ \sup\limits_{v \in S^{d - 1}}|\E_{\rho_{v}}\operatorname{Med}(\langle X_1, \theta \rangle, \ldots, \langle X_{N}, \theta \rangle) - \langle \mu, v\rangle|
    \\
    &\le 2\sup\limits_{v \in S^{d - 1}}|\E_{\rho_{v}}\operatorname{Med}(\langle X_1, \theta \rangle, \ldots, \langle X_{N}, \theta \rangle) - \langle \mu, v\rangle|
    \\
    &=2\sup\limits_{v \in S^{d - 1}}|\E_{\rho_{v}}\operatorname{Med}(\langle X_1-\mu, \theta \rangle, \ldots, \langle X_{N} - \mu, \theta \rangle)|.
    \end{align*}
We need to bound the last quantity. From now on we can assume without loss of generality that $\mu = 0$. Assume $Y_1, \ldots, Y_{N}$ is an uncontaminated sample of zero mean independent Gaussians with covariance $\Sigma$.  That is, at most $\eps N$ elements among $X_1, \ldots, X_{N}$ are different from their $Y_1, \ldots, Y_{N}$ counterparts. Observe that the sample median of projections of the contaminated sample on any direction cannot be too far away from $1/2 \pm \eps$ quantiles of the corresponding projections for the uncontaminated sample. Formally, assuming that both the sample median and $1/2\pm \eps$ sample quantiles are unique, we have $\rho_v$-almost surely
\begin{align*}
\operatorname{Quant}_{\frac{1}{2} - \eps}(\langle Y_1, \theta \rangle, \ldots, \langle Y_{N}, \theta \rangle) &\le \operatorname{Med}(\langle X_1, \theta \rangle, \ldots, \langle X_{N}, \theta \rangle) 
\\
&\le\operatorname{Quant}_{\frac{1}{2}  + \eps}(\langle Y_1, \theta \rangle, \ldots, \langle Y_{N}, \theta \rangle),
\end{align*}
and thus, taking the expectation with respect to $\rho_v$, we have
\begin{align*}
\E_{\rho_v}\operatorname{Quant}_{\frac{1}{2} - \eps}(\langle Y_1, \theta \rangle, \ldots, \langle Y_{N}, \theta \rangle) &\le \E_{\rho_v}\operatorname{Med}(\langle X_1, \theta \rangle, \ldots, \langle X_{N}, \theta \rangle) 
\\
&\le \E_{\rho_v}\operatorname{Quant}_{\frac{1}{2}  + \eps}(\langle Y_1, \theta \rangle, \ldots, \langle Y_{N}, \theta \rangle).
\end{align*}
Therefore, we have for any $v \in S^{d-1}$,
\begin{align*}
\left|\E_{\rho_{v}}\operatorname{Med}(\langle X_1, \theta \rangle, \ldots, \langle X_{N}, \theta \rangle)\right|
&\le \left|\E_{\rho_{v}}\operatorname{Quant}_{\frac{1}{2} + \eps}(\langle Y_1, \theta \rangle, \ldots, \langle Y_{N}, \theta \rangle)\right|
\\
&\qquad +\left|\E_{\rho_{v}}\operatorname{Quant}_{\frac{1}{2} - \eps}(\langle Y_1, \theta \rangle, \ldots, \langle Y_{N}, \theta \rangle)\right|.
\end{align*}
Both terms will be analyzed similarly. We only analyze the first one.
Observe that due to the spherical symmetry, we have that $S_{N} = \{\langle Y_1, \theta \rangle/\sqrt{\theta^{\top}\Sigma\theta}, \ldots, \langle Y_{N}, \theta \rangle/\sqrt{\theta^{\top}\Sigma\theta}\}$ consists of independent standard Gaussian random variables (in our case, $\theta \neq 0$ almost surely). We have
\begin{align*}
\left|\E_{\rho_{v}}\operatorname{Quant}_{\frac{1}{2} + \eps}(\langle Y_1, \theta \rangle, \ldots, \langle Y_{N}, \theta \rangle)\right| &\le\left|\E_{\rho_{v}}\sqrt{\theta^{\top}\Sigma\theta}\left(\operatorname{Quant}_{\frac{1}{2} + \eps}(S_{N}) - \E\operatorname{Quant}_{\frac{1}{2} + \eps}(S_{N})\right)\right|
\\
&\qquad+\left|\E_{\rho_v}\sqrt{\theta^{\top}\Sigma\theta}\cdot(\E\operatorname{Quant}_{\frac{1}{2} + \eps}(S_{N}) - \Phi^{-1}(1/2 + \eps))\right|
\\
&\qquad+\E_{\rho_{v}}\sqrt{\theta^{\top}\Sigma\theta}\cdot\Phi^{-1}(1/2 + \eps)
\\
& = (I) + (II) + (III).
\end{align*}
To upper bound $(I)$ we want to apply Lemma \ref{lem:pacbayes}. Fix $\lambda > 0$ and let $\gamma$ be a multivariate Gaussian distribution in $\mathbb{R}^d$ with zero mean and covariance $\beta^{-1}I_d$. The standard formula implies $\mathcal{KL}(\rho_v, \gamma) = \beta/2$. Thus, by Lemma \ref{lem:pacbayes} we have simultaneously for all $v \in S^{d-1}$, with probability at least $1 - \delta$,
\begin{align*}
&\lambda\E_{\rho_v}\sqrt{\theta^{\top}\Sigma\theta}\left(\operatorname{Quant}_{\frac{1}{2} + \eps}(S_{N}) -  \E\operatorname{Quant}_{\frac{1}{2} + \eps}(S_{N})\right) 
\\
&\quad\le \E_{\rho_v}\log\E\exp\left(\lambda\sqrt{\theta^{\top}\Sigma\theta}\left(\operatorname{Quant}_{\frac{1}{2} + \eps}(S_{N}) -  \E\operatorname{Quant}_{\frac{1}{2} + \eps}(S_{N})\right)\right) + \beta/2 + \log(1/\delta).
\end{align*}
Since centering multiplies the $\psi_2$-norm by at most an absolute constant factor (see \eg, \cite[Lemma 2.6.8]{Vershynin2016HDP}), we have by Lemma \ref{lem:subGquantiles}, for some absolute constant $c_1 > 0$,
\[
\left\|\operatorname{Quant}_{\frac{1}{2} + \eps}(S_{N}) -  \E\operatorname{Quant}_{\frac{1}{2} + \eps}(S_{N})\right\|_{\psi_2} \le \frac{c_1}{\sqrt{N}}.
\]
Thus, by \cite[Proposition 2.5.2, (v)]{Vershynin2016HDP} (conditioned on $\theta$, we take $\lambda\sqrt{\theta^{\top}\Sigma\theta}$ instead of $\lambda$ in that result), we have for some {absolute constant} $c_2>0$,
\begin{align*}
\E_{\rho_v}\log\E\exp\left(\lambda\sqrt{\theta^{\top}\Sigma\theta}\left(\operatorname{Quant}_{\frac{1}{2} + \eps}(S_{N}) -  \E\operatorname{Quant}_{\frac{1}{2} + \eps}(S_{N})\right)\right) &\le \frac{\E_{\rho_v}c_2\lambda^2\theta^{\top}\Sigma\theta}{N}
\\
&= \frac{c_2\lambda^2(v^{\top}\Sigma v + \beta^{-1}\tr(\Sigma))}{N}
\\
&\le \frac{11c_2\lambda^2\|\Sigma\|}{N},
\end{align*}
where the last lines are based on a direct computation and our choice of $\beta$ (we have $\beta^{-1} \le 10\|\Sigma\|/\tr(\Sigma)$). Optimizing the bound on $\lambda\E_{\rho_v}\sqrt{\theta^{\top}\Sigma\theta}\left(\operatorname{Quant}_{\frac{1}{2} + \eps}(S_{N}) -  \E\operatorname{Quant}_{\frac{1}{2} + \eps}(S_{N})\right)$ with respect to $\lambda > 0$ and since $\beta \le 10\tr(\Sigma)/\|\Sigma\|$, we obtain that uniformly over $S^{d -1}$,
\[
\E_{\rho_v}\sqrt{\theta^{\top}\Sigma\theta}\left(\operatorname{Quant}_{\frac{1}{2} + \eps}(S_{N}) -  \E\operatorname{Quant}_{\frac{1}{2} + \eps}(S_{N})\right) \le c_3\sqrt{\frac{\tr(\Sigma) + \|\Sigma\|\log(1/\delta)}{N}},
\]
where $c_3 > 0$ is an absolute constant. Repeating the proof for $\lambda < 0$ and using the union bound, we extend this bound to the upper bound $(I)$. 

We now focus on bounding $(II)$. Observe that for any scalar $C$, we have $\|C\|_{\psi_2} = |C|/\sqrt{\log 2}$. Using this observation, together with Jensen's inequality and Lemma \ref{lem:subGquantiles} we have for some $c_4 > 0$,
\begin{align*}
\left|\E_{\rho_v}\sqrt{\theta^{\top}\Sigma\theta}\cdot(\E\operatorname{Quant}_{\frac{1}{2} + \eps}(S_{N}) - \Phi^{-1}(1/2 + \eps))\right| &\le \frac{c_4\sqrt{\log 2}\cdot\E_{\rho_v}\sqrt{\theta^{\top}\Sigma\theta}}{\sqrt{N}}
\\
&\le \frac{c_4\sqrt{\log 2}\cdot\sqrt{11\|\Sigma\|}}{\sqrt{N}}.
\end{align*}
Finally, we bound $(III)$. First, we notice that the function $\Phi^{-1}(\cdot)$ is locally Lipschitz on a closed interval $[1/2-\eps, 1/2+\eps]$ for $\eps \in [0, 1/4]$. We compute and bound the local Lipschitz constant of $\Phi^{-1}$ as follows
\[
\frac{d}{dx} \Phi^{-1}(1/2 + x) = \sqrt{2\pi} \exp\Big((\Phi^{-1}(1/2+x))^2/{2}\Big) \le \sqrt{2\pi} \exp\Big((\Phi^{-1}(3/4))^2/{2}\Big) \le 4,
\]
where we used the fact that the function $ \exp(\Phi^{-1}(\cdot)^2)$ is increasing. Since the standard Gaussian distribution is symmetric, we have $\Phi^{-1}(1/2) = 0$. Hence, the bound for $\Phi^{-1}(1/2 + \eps)$ reads as
\[
\Phi^{-1}(1/2 + \eps) = \Phi^{-1}(1/2 + \eps) - \Phi^{-1}(1/2) \le 4 \eps.
\]
Therefore, we have
\[
\E_{\rho_{v}}\sqrt{\theta^{\top}\Sigma\theta}\cdot\Phi^{-1}(1/2 + \eps) \le 4\eps\sqrt{11\|\Sigma\|}.
\]
Combining the above bounds concludes the proof.
\end{proof}
We are now ready to prove our second main result.
\begin{proof}[Proof of Theorem \ref{thm:covarianceestimation}] Recall that $H$ is a covariance matrix of $\rho_v$ and does not depend on direction $v \in S^{d-1}$.
Observe that $\E_{\rho_v}\theta^{\top}\Sigma\theta = v^{\top}\Sigma v + \tr(\Sigma H)$. Moreover, since our choice of parameters implies $p \ge 1/2$ in \eqref{eq:measureforkl}, we have
\[
H \preceq 2\beta^{-1}I_d, \quad \textrm{and} \quad \|H\| \le 2\beta^{-1}.
\]
We also observe that $\mathbf{r}(G) = \tr(G)/\|G\| \le 100\mathbf{r}(\Sigma)$. Using the triangle inequality, as well as the definition of our estimator combined with the definition of the set $\mathcal H$ from \eqref{eq:setH}, we have
\begin{align*}
\|\widehat{\Sigma} - \Sigma\| &= \sup\limits_{v \in S^{d - 1}}\left|v^{\top}{\Sigma} v - v^{\top}\widehat{\Sigma} v\right|
\\
&\le \sup\limits_{v \in S^{d - 1}}\left|\E_{\rho_v}\theta^{\top}({\Sigma} - \widehat{\Sigma})\theta\right| + |(\tr(\widehat{\Sigma}H) - \alpha(H)) - (\tr({\Sigma}H)-\alpha(H))|
\\
&\le \sup\limits_{v \in S^{d - 1}}\left|\E_{\rho_v}\theta^{\top}({\Sigma} - \widehat{\Sigma})\theta\right| +c\tr((\Sigma + \widehat{\Sigma})H)\left(\sqrt{\frac{\max\{\mathbf{r}(G), \mathbf{r}(\Sigma)\} + \log(1/\delta)}{N}} + \eps\right)
\\
&\le \sup\limits_{v \in S^{d - 1}}\left|\E_{\rho_v}\theta^{\top}({\Sigma} - \widehat{\Sigma})\theta\right| +2\beta^{-1}c\tr(\Sigma + 10G)\left(\sqrt{\frac{100\mathbf{r}(\Sigma) + \log(1/\delta)}{N}} + \eps\right)
\\
&\le \sup\limits_{v \in S^{d - 1}}\left|\E_{\rho_v}\theta^{\top}({\Sigma} - \widehat{\Sigma})\theta\right| +202\beta^{-1}c\tr(\Sigma)\left(\sqrt{\frac{100\mathbf{r}(\Sigma) + \log(1/\delta)}{N}} + \eps\right).
\end{align*}
Since $\beta^{-1} \le 10\mathbf{r}(\Sigma)$, the last term in the last inequality is not larger than the rate of convergence in the statement of Theorem \ref{thm:covarianceestimation}. We now can focus only on bounding the first term in the last line of the inequalities from the previous display. We first need some auxiliary computations. Using the definition of the set $\mathcal H$, we have $\rho_v$-almost surely
\begin{align*}
\sqrt{\theta^{\top}{\Sigma} \theta} + \sqrt{\theta^{\top}\widehat{\Sigma} \theta}&\le \sqrt{2v^{\top}{\Sigma} v + 2(\theta - v)^{\top}{\Sigma} (\theta - v)} + \sqrt{2v^{\top}\widehat{\Sigma} v + 2(\theta - v)^{\top}\widehat{\Sigma} (\theta - v)}
\\
&\le \sqrt{2\|\Sigma\| + 20(\theta - v)^{\top}{G} (\theta - v)} + \sqrt{20\omega + 20(\theta - v)^{\top}G (\theta - v)}
\\
&\le \sqrt{2\|\Sigma\| + 20r^2} + \sqrt{20\omega + 20r^2}
\\
&\le c_1\sqrt{\|\Sigma\|},
\end{align*}
where $c_1 > 0$ is some absolute constant.
This implies the following lines
\begin{align*}
&\sup\limits_{v \in S^{d - 1}}\left|\E_{\rho_v}\theta^{\top}({\Sigma} - \widehat{\Sigma})\theta\right| 
\\
&\le \sup\limits_{v \in S^{d - 1}}\E_{\rho_v}\left|\sqrt{\theta^{\top}{\Sigma} \theta} - \sqrt{\theta^{\top}\widehat{\Sigma} \theta}\right|\left|\sqrt{\theta^{\top}{\Sigma} \theta} + \sqrt{\theta^{\top}\widehat{\Sigma} \theta}\right|
\\
&\le c_1\sqrt{\|\Sigma\|}\sup\limits_{v \in S^{d - 1}}\E_{\rho_v}\left|\sqrt{\theta^{\top}{\Sigma} \theta} - \sqrt{\theta^{\top}\widehat{\Sigma} \theta}\right|
\\
&\le c_1\sqrt{\|\Sigma\|}\sup\limits_{v \in S^{d - 1}}\E_{\rho_v}\left|\operatorname{Med}\left(|\langle X_1, \theta\rangle|, \ldots, |\langle X_N, \theta\rangle|\right)/(\Phi^{-1}(3/4)) - \sqrt{\theta^{\top}\widehat{\Sigma}\theta}\right|
\\
&\qquad+ c_1\sqrt{\|\Sigma\|}\sup\limits_{v \in S^{d - 1}}\E_{\rho_v}\left|\operatorname{Med}\left(|\langle X_1, \theta\rangle|, \ldots, |\langle X_N, \theta\rangle|\right)/(\Phi^{-1}(3/4)) - \sqrt{\theta^{\top}\Sigma\theta}\right|
\\
&\le2c_1\sqrt{\|\Sigma\|}\sup\limits_{v \in S^{d - 1}}\E_{\rho_v}\left|\operatorname{Med}\left(|\langle X_1, \theta\rangle|, \ldots, |\langle X_N, \theta\rangle|\right)/(\Phi^{-1}(3/4)) - \sqrt{\theta^{\top}\Sigma\theta}\right|,
\end{align*}
where in the last line we used the definition of $\widehat{\Sigma}$ and that $\Sigma \in \mathcal H$.
We focus on upper bounding the last expression. 
Let $Y_1, \ldots, Y_N$ be the uncontaminated version of our $\varepsilon$-contaminated sample. Using the same argument as in the proof of Theorem \ref{thm:maintheorem}, we have
\begin{align*}
&\sup\limits_{v \in S^{d - 1}}\E_{\rho_v}\left|\operatorname{Med}\left(|\langle X_1, \theta\rangle|, \ldots, |\langle X_N, \theta\rangle|\right) - \Phi^{-1}(3/4)\sqrt{\theta^{\top}\Sigma\theta}\right|
\\
&\le\sup\limits_{v \in S^{d - 1}}\E_{\rho_v}\left|\operatorname{Quant}_{\frac{1}{2} + \eps}\left(|\langle Y_1, \theta\rangle|, \ldots, |\langle Y_N, \theta\rangle|\right) - \Phi^{-1}(3/4)\sqrt{\theta^{\top}\Sigma\theta}\right|
\\
&\qquad+ \sup\limits_{v \in S^{d - 1}}\E_{\rho_v}\left|\operatorname{Quant}_{\frac{1}{2} - \eps}\left(|\langle Y_1, \theta\rangle|, \ldots, |\langle Y_N, \theta\rangle|\right) - \Phi^{-1}(3/4)\sqrt{\theta^{\top}\Sigma\theta}\right|.
\end{align*}
We only analyze the first term. Observe that due to the spherical symmetry, we have that $S_{N} = \{|\langle Y_1, \theta \rangle|/\sqrt{\theta^{\top}\Sigma\theta}\ , \ldots,\  |\langle Y_{N}, \theta \rangle|/\sqrt{\theta^{\top}\Sigma\theta}\}$ consists of independent half-normal random variables (in our case, $\theta \neq 0$ almost surely). By the triangle inequality, we have
\begin{align*}
&\sup\limits_{v \in S^{d - 1}}\E_{\rho_v}\sqrt{\theta^{\top}\Sigma\theta}\left|\operatorname{Quant}_{\frac{1}{2} + \eps}\left(S_N\right) - \Phi^{-1}(3/4)\right|
\\
&\le \sup\limits_{v \in S^{d - 1}}\E_{\rho_v}\sqrt{\theta^{\top}\Sigma\theta}\left(\left|\operatorname{Quant}_{\frac{1}{2} + \eps}\left(S_N\right) -\Phi^{-1}_{\operatorname{H}}(1/2 + \eps)\right| - \E\left|\operatorname{Quant}_{\frac{1}{2} + \eps}\left(S_N\right) -\Phi^{-1}_{\operatorname{H}}(1/2 + \eps)\right|\right)
\\
&\qquad+\sup\limits_{v \in S^{d - 1}}\E_{\rho_v}\E\sqrt{\theta^{\top}\Sigma\theta}\left|\operatorname{Quant}_{\frac{1}{2} + \eps}\left(S_N\right) -\Phi^{-1}_{\operatorname{H}}(1/2 + \eps)\right|
\\
&\qquad+\sup\limits_{v \in S^{d - 1}}\E_{\rho_v}\sqrt{\theta^{\top}\Sigma\theta}\left|\Phi^{-1}(3/4) -\Phi^{-1}_{\operatorname{H}}(1/2 + \eps)\right|
\\
&= (I) + (II) + (III).
\end{align*}
We want to apply Lemma \ref{lem:pacbayes} to control $(I)$. Fix $\lambda > 0$ and let $\gamma$ be a multivariate Gaussian distribution in $\mathbb{R}^d$ with zero mean and covariance $\beta^{-1}I_d$. Lemma \ref{lem:kllemma} implies that for our choice of parameters $\mathcal{KL}(\rho_v, \gamma) \le \log(2)+\beta/2$. Denote 
\[
Q(S_N) = \left|\operatorname{Quant}_{\frac{1}{2} + \eps}\left(S_N\right) -\Phi^{-1}_{\operatorname{H}}(1/2 + \eps)\right| - \E\left|\operatorname{Quant}_{\frac{1}{2} + \eps}\left(S_N\right) -\Phi^{-1}_{\operatorname{H}}(1/2 + \eps)\right|.
\]
Observe that conditioned on $\theta$, the random variable $Q(S_N)$ is a centered version of the random variable $|\operatorname{Quant}_{\frac{1}{2} + \eps}\left(S_N\right) -\Phi^{-1}_{\operatorname{H}}(1/2 + \eps)|$ whose $\|\cdot\|_{\psi_2}$ is controlled by Lemma \ref{lem:orderstatsofhalfnormal}. Since centering multiplies the $\psi_2$-norm by at most an absolute constant factor, we have (conditioned on $\theta$) that $\left\|Q(S_N)\right\|_{\psi_2} \le \frac{c_2}{\sqrt{N}}$ for some absolute constant $c_2 > 0$. By Lemma \ref{lem:pacbayes} we have, simultaneously for all $v \in S^{d-1}$, with probability at least $1 - \delta$,
\begin{align*}
&\lambda\E_{\rho_{v}}\sqrt{\theta^{\top}\Sigma\theta}\ Q(S_N) \le\E_{\rho_v}\log\E\exp\left(\lambda\sqrt{\theta^{\top}\Sigma\theta}\ Q(S_N)\right) +  \beta/2 + \log(2/\delta).
\end{align*}
Thus, by \cite[Proposition 2.5.2, (v)]{Vershynin2016HDP} (conditioned on $\theta$, we take $\lambda\sqrt{\theta^{\top}\Sigma\theta}$ instead of $\lambda$ in that result), repeating the lines of the proof of Theorem \ref{thm:maintheorem}, we have for some absolute constants $c_3, c_4>0$,
\[
\E_{\rho_v}\log\E\exp\Big(\lambda\sqrt{\theta^{\top}\Sigma\theta}\ Q(S_N)\Big) \le \E_{\rho_v}\frac{c_3\lambda^2\theta^{\top}\Sigma\theta}{N} \le \frac{c_4\lambda^2\|\Sigma\|}{N}.
\]
Combining the bounds and optimizing with respect to $\lambda$, we have simultaneously for all $v \in S^{d-1}$, with probability at least $1 - \delta$,
\[
(I) \le c_5{\|\Sigma\|^{1/2}}\Bigg(\sqrt{\frac{\mathbf{r}(\Sigma)}{N}} + \sqrt{\frac{\log(1/\delta)}{N}}\Bigg),
\]
where $c_5 > 0$ is some absolute constant. We now bound the term $(II)$. Similarly to the proof of Theorem \ref{thm:maintheorem}, we use Lemma \ref{lem:orderstatsofhalfnormal} to get, for some absolute constant $c_6 > 0$, the following bound
\[
\E_{\rho_v}\E\sqrt{\theta^{\top}\Sigma\theta}\left|\operatorname{Quant}_{\frac{1}{2} + \eps}\left(S_N\right) -\Phi^{-1}_{\operatorname{H}}(1/2 + \eps)\right| \le c_6\sqrt{\frac{\|\Sigma\|}{N}}.
\]
To bound $(III)$ we first observe that $\Phi^{-1}_{\operatorname{H}}(1/2) = \Phi^{-1}(3/4)$. Now we show that the difference $\Phi_{\operatorname{H}}^{-1}(1/2 + \eps) - \Phi^{-1}_{\operatorname{H}}(1/2)$ is bounded by $\eps$ (up to multiplicative constant) for $\eps \in [0, 1/4]$. Similarly to the arguments used in the proof of Theorem \ref{thm:maintheorem} for the quantile function of standard Gaussian distribution, we compute and bound the derivative of $\Phi_{\operatorname{H}}^{-1}(1/2+x)$ when $x \in [0, 1/4]$ as follows 
\[
\frac{d}{dx} \Phi_{\operatorname{H}}^{-1}(1/2 + x) = \sqrt{\frac{\pi}{2}}
\exp\Big((\Phi^{-1}_{\operatorname{H}}(1/2+x))^2/{2}\Big) \le \sqrt{\frac{\pi}{2}} \exp\Big((\Phi^{-1}_{\operatorname{H}}(3/4))^2/{2}\Big) \le 3.
\]
Therefore, we have for some $c_7 > 0$,
\[
\E_{\rho_{v}}\sqrt{\theta^{\top}\Sigma\theta}\cdot\big|\Phi_{\operatorname{H}}^{-1}(1/2 + \eps) - \Phi_{\operatorname{H}}^{-1}(1/2)\big|\le c_7\eps\sqrt{\|\Sigma\|}.
\]
Combining the obtained bounds, we complete the proof.
\end{proof}
\paragraph{Statistical optimality of our estimators.}
We shortly discuss the claimed optimality of our bounds. The optimality results follow immediately from existing lower bounds. The bounds in \cite[Theorem 2.2 and Theorem 3.2]{chen2018robust} show that Theorem \ref{thm:maintheorem} and Theorem \ref{thm:covarianceestimation} both have the optimal dependence on the contamination level with correct dimension-free parametric rate. For covariance estimation, the optimality of the remaining terms is described in detail in \cite[Section 5]{abdalla2022covariance}. Matching lower bounds for the mean estimation problem
are shown in \cite{lugosi2019near}.

\section{Tuning the unknown parameters}
\label{sec:tuning}
Our focus is now on tuning a few parameters used in our estimators. For the sake of simplicity, we assume that either $\eps$ is known exactly or a known upper bound $\varepsilon_0$ is available, such that $\varepsilon \le \varepsilon_0 < 1/2$. This is a standard assumption in the literature \cite{lugosi2021robust}. Observe that at least in mean estimation the value of $\eps$ is only used to tune the parameter $\beta$.
The most standard approach to estimating other parameters is the sample-splitting idea. One splits the sample into several independent blocks of equal sizes. For each block, we can bound the number of contaminated points. This will allow us to state our result for any $\eps \in [0, c]$, where $c$ is some small enough absolute constant. An interesting aspect of our analysis is that we can tune different parameters on the same sample. We will now discuss this in more details.

\paragraph{Handling the dependencies.} It is clear that in the strong contamination setup, the adversary can make the aforementioned blocks dependent. That is, the outliers in any sub-sample may depend on the entire sample. Some authors assume implicitly that the splitting of the sample results in independent subsamples. For example, the analysis of the trimmed-mean estimator in \cite[Theorem 1]{lugosi2021robust} uses this independence, which holds, for example, in Huber's contamination model, but is not true in the general strong contamination model. Taking care of the sample splitting step in the strong contamination model requires some additional stability-type analysis. We refer to \cite[Section 6]{diakonikolas2022outlier} as an example of this approach.

We now show that our approach allows one to tune the parameters on the same sample. Thus, our result is valid in the strong contamination model without additional assumptions. For clarity, we only focus on the mean estimation problem. Assume we are given an $\eps$-contaminated sample of size $N$. We denote it by $S_{N}$. Given $S_{N}$, we first find an integer $\beta = \beta(S_N)$ satisfying, with probability at least $1 - \delta/2$,
\begin{equation}
\label{eq:betaeq}
\mathbf{r}(\Sigma)/10 \le \beta(S_N) \le 10\mathbf{r}(\Sigma).
\end{equation}
We then compute our estimator defined in \eqref{eq:ourestimator} on the same sample $S_N$ with $\beta = \beta(S_N)$. Denote the event where \eqref{eq:betaeq} holds by $E$. We show that due to the nature of Lemma \ref{lem:pacbayes}, this dependence does not lead to additional technical issues. First, observe that since $\beta$ is an integer, we can use the union bound over at most $10\mathbf{r}(\Sigma)$ prior Gaussian distributions $\gamma$ to handle potential dependence of $\beta$ on $S_{N}$. One can verify that this application of the union bound does not change the bound of Theorem \ref{thm:maintheorem}. Importantly, the result of Lemma \ref{lem:pacbayes} is uniform with respect to the posterior distribution $\rho_v$ and allows $\beta$ to depend on the sample as long as $\mathcal{KL}(\rho_v, \gamma) = \beta(S_N)/2 \le 5\mathbf{r}(\Sigma)$, which holds on the event $E$. Finally, one can easily verify that, on the same event $E$, the desired upper bound on the term
\[
\E_{\rho_v}\log\E\exp\left(\lambda\sqrt{\theta^{\top}\Sigma\theta}\left(\operatorname{Quant}_{\frac{1}{2} \pm \eps}(S_{N}) -  \E\operatorname{Quant}_{\frac{1}{2} \pm \eps}(S_{N})\right)\right),
\]
appearing in the proof of Theorem \ref{thm:maintheorem} is not affected by the fact that $\beta = \beta(S_N)$. This argument allows us to use $\beta(S_N)$ in our estimator. 
 
Similar ideas can also be applied in the covariance estimation setup. To avoid unnecessary technicalities, we assume that for covariance estimation we can indeed split the sample into several blocks and the adversary is not allowed to create dependencies between these blocks. This covers many standard contamination models, including Huber's $\varepsilon$-contamination model. 

\paragraph{Estimating $\beta$ and $\omega$.}
This step follows from existing results. In particular, in the Gaussian case Proposition 6 in \cite{abdalla2022covariance} provides an estimator $\omega$ satisfying $\|\Sigma\|/4 \le \omega \le 4\|\Sigma\|$ whenever $N \ge c(\mathbf{r}(\Sigma) + \log(1/\delta))$, where $c > 0$ is an absolute constant. We also need to estimate $\tr(\Sigma)$. This problem reduces to mean estimation. The linear dependence on $\eps$ will not play any role since we only need to know  $\tr(\Sigma)$ up to a multiplicative constant factor. In particular, one can use any \emph{sub-Gaussian mean estimator} in $\mathbb{R}$ (see \cite{lugosi2019mean} for the exact definition) that is tolerant to strong contamination and gives a $\sqrt{\eps}$-dependence on the contamination level to find $\tau$ satisfying $\tr(\Sigma)/2 \le \tau \le 2\tr(\Sigma)$, whenever $\eps$ is small enough and $N \ge c \log(1/\delta)$. This allows us to find an integer $\beta$ satisfying \eqref{eq:betaeq}.

\paragraph{Constructing the matrix $G$.}

We discuss how to construct a positive semi-definite matrix $G$, satisfying 
\begin{equation}
\label{eq:secondmatrixg}
\Sigma \preceq 10G, \quad \textrm{and} \quad \tr(G) \le 10\tr(\Sigma).
\end{equation}
The following result allows us to construct such a matrix efficiently whenever $N \ge c(d + \log(1/\delta))$, where $c > 0$ is some absolute constant.
\begin{proposition}
There are absolute constants $c, c_1 > 0$ such that the following holds. Assume that $X$ is a Gaussian zero mean vector in $\mathbb{R}^d$ with covariance $\Sigma$. Let $X_1, \ldots, X_N$ be an $\eps$-contaminated set of independent copies of $X$. Fix $\delta \in (0, 1)$. Assume that $\eps \le c$ and $N \ge c_1(d + \log(1/\delta))$. Then, with probability at least $1 - \delta$, simultaneously for all $I^{\prime} \subseteq [N]$ such that $|I^{\prime}| = cN$, we have 
\[
\Sigma \preceq \frac{10}{N}\sum\limits_{i \in [N]\setminus I^{\prime}}X_iX_i^{\top}.
\]
Moreover, on the same event, there exists $I \subseteq [N]$ such that $|I| = cN$, and
\[
\frac{1}{N}\sum\limits_{i \in [N]\setminus I}\|X_i\|^2 \le 10\tr(\Sigma).
\]
\end{proposition}

This result implies immediately that the matrix $G = \frac{1}{N}\sum\nolimits_{i \in [N]\setminus I}X_iX_i^{\top}$ satisfies the desired property \eqref{eq:secondmatrixg}. In order to find this set, one only needs to find a set $I$ of size $\eps N$ such that $\sum\nolimits_{i \in [N]\setminus I}\|X_i\|^2 \le 10N\tr(\Sigma)$. This can be done simply by removing the $\eps N$ vectors with the largest norms. 

\begin{proof}
Without loss of generality, we assume that $cN$ is an integer. Fix any $I \subset [N]$ of size $2cN$. Let $Y_1, \ldots, Y_N$ denote an uncontaminated sample. 
The total number of such subsets is upper bounded by $\binom{N}{2cn} \le \left(\frac{2e}{c}\right)^{2cN}$.
By the bound of Oliveira \cite[Theorem 4.1 with $\mathrm{h} = 3$]{oliveira2016lower} and the union bound over all sets $I$ of size $2cN$, we have
\[
\Sigma\left(1 - 27\sqrt{\frac{d + 4cN\log(\frac{2e}{c})+  2\log(2/\delta)}{N - 2cN}}\right) \preceq \frac{1}{N - 2cN}\sum\limits_{i \in [N]\setminus I}Y_iY_i^{\top}.
\]
When $c$ is small enough and $N \ge c_1(d + \log(1/\delta))$ for large enough $c_1 > 0$, on the same event, we have 
\[
\Sigma \preceq \frac{10}{N}\sum\limits_{i \in [N]\setminus I}Y_iY_i^{\top}.
\]
Observe that since each term $Y_iY_i^{\top}$ is a positive semi-definite matrix and $\eps \le c$, we have that for any $I^{\prime}$ of size $cN$, there is a set $I$ of size $2cN$  such that
\[
\sum\nolimits_{i \in [N]\setminus I}Y_iY_i^{\top} \preceq \sum\nolimits_{i \in [N]\setminus I^{\prime}}X_iX_i^{\top}.
\]
Indeed, to build such a set $I$ we consider the union of the set of contaminated points with the set $I^{\prime}$ (we can add any additional elements if the cardinality of this union is less than $2cN$).
This implies that under our assumption for all $I^{\prime} \subset [N]$ of size $cN$, with probability at least $1 - \delta$,
\[
\Sigma \preceq \frac{10}{N}\sum\nolimits_{i \in [N]\setminus I^{\prime}}X_iX_i^{\top}.
\]
We now consider the second part of the statement. Combining the Gaussian concentration inequality \cite[Example 5.7]{boucheron2013concentration} and \cite[Proposition 2.5.2]{Vershynin2016HDP}, we get that there is an absolute constant $c_2> 0$ such that
\[
\bigl\|\|X\| - \E\|X\|\bigr\|_{\psi_2} \le c_2\sqrt{\|\Sigma\|}.
\]
It is now standard to verify that $\bigl\|\|X\|^2 - \E\|X\|^2\bigr\|_{\psi_1} \le c_3\|\Sigma\|$, where $c_3>0$ is an absolute constant. By the Bernstein inequality \cite[Theorem 2.8.1]{Vershynin2016HDP} and the union bound, simultaneously for all $I \subset [N], |I| = cN$, with probability at least $1 - \delta$, it holds for some absolute constant $c_4 > 0$ that
\begin{align*}
\sum\limits_{i \in [N]\setminus I}\|Y_i\|^2 &\le N\tr(\Sigma) +c_4\|\Sigma\|\left(\sqrt{N(\log(1/\delta) + cN\log(e/c))} + \log(1/\delta) + cN\log(e/c)\right)
\\
&\le 10N\tr(\Sigma).
\end{align*}
The last inequality holds provided that $c$ is small enough and $c_1$ is large enough. We choose $I$ to be the set corresponding to the set of contaminated points. For this set $I$, on the same event, we have 
\[
\frac{1}{N}\sum\limits_{i \in [N]\setminus I}\|X_i\|^2 = \frac{1}{N}\sum\limits_{i \in [N]\setminus I}\|Y_i\|^2 \le 10\tr(\Sigma).
\]
The claim follows by the union bound.
\end{proof}

\paragraph{Estimating $\alpha(H)$.}
We conclude by the analysis of a real-valued parameter $\alpha = \alpha(H)$, defined in \eqref{eq:alphah}.
In what follows, $H$ is a known positive semi-definite matrix. When allowing slightly sub-optimal dependence on $\eps$, we can use the analysis of the trimmed mean estimator in $\mathbb{R}$ (see \cite[Theorem 1]{lugosi2021robust}). Unfortunately, the analysis becomes more complicated when the linear dependence on the contamination level is of interest. Recall that we are interested in finding $\alpha=\alpha(H)$ such that, with probability at least $1-\delta$,
\[
|\alpha - \tr(\Sigma H)| \le c\tr(\Sigma H)\left(\sqrt{\frac{\mathbf{r}(\Sigma) + \log(1/\delta)}{N}} + \eps\right).
\]
We present an estimator that achieves this error rate in almost any interesting regime. More precisely, we will either make an additional assumption that $\delta \ge \exp(-\sqrt{\mathbf{r}(\Sigma)})$, or that $\log d \le \mathbf{r}(\Sigma)$. In what follows, $e_1, \ldots, e_d$ denotes the standard basis in $\mathbb{R}^d$.
\begin{proposition}
There are absolute constants $c, c_1, c_2 > 0$ such that the following holds. Assume that $X$ is a Gaussian zero mean vector in $\mathbb{R}^d$ with covariance $\Sigma$. Let $X_1, \ldots, X_N$ be an $\eps$-contaminated set of independent copies of $X$. Fix $\delta \in (0, 1)$. Assume that $\eps \le c$ and $N \ge c_1\log(1/\delta)$. Then, with probability at least $1 - \delta$, it holds
\begin{align}
&\left|(\Phi^{-1}(3/4))^{-2}\sum\limits_{i = 1}^d\operatorname{Med}\left(\langle e_i, H^{1/2}X_1 \rangle^2, \ldots,\langle e_i, H^{1/2}X_N \rangle^2\right) - \tr(\Sigma H)\right| \nonumber
\\
&\qquad\qquad\le c_1\tr(\Sigma H)\left(\frac{\log(1/\delta)}{\sqrt{N}} + \eps\right).
\label{eq:firstineq}
\end{align}
If additionally, $N \ge c_2(\log d + \log(1/\delta))$, then on the same event, it holds
\begin{align*}
&\left|(\Phi^{-1}(3/4))^{-2}\sum\limits_{i = 1}^d\operatorname{Med}\left(\langle e_i, H^{1/2}X_1 \rangle^2, \ldots,\langle e_i, H^{1/2}X_N \rangle^2\right) - \tr(\Sigma H)\right|
\\
&\qquad\qquad\le c_1\tr(\Sigma H)\left(\sqrt{\frac{\log d + \log(1/\delta)}{N}} + \eps\right).
\end{align*}
\end{proposition}

\begin{proof} Let $Y_1, \ldots, Y_N$ denote the uncontaminated sample, and let $Y$ be a zero mean Gaussian in $\mathbb{R}^d$ with covariance $\Sigma$.  Since $\tr(\Sigma H) = \tr(H^{1/2}\Sigma H^{1/2})$, by triangle inequality and the arguments of the proof of Theorem \ref{thm:maintheorem}, we have
\begin{align*}
&\left|\sum\limits_{i = 1}^d\operatorname{Med}\left(\langle e_i, H^{1/2}X_1 \rangle^2, \ldots,\langle e_i, H^{1/2}X_N \rangle^2\right) - (\Phi^{-1}(3/4))^{2}\tr(\Sigma H)\right|
\\
&\le \sum\limits_{i = 1}^d\left|\operatorname{Med}\left(\langle e_i, H^{1/2}X_1 \rangle^2, \ldots,\langle e_i, H^{1/2}X_N \rangle^2\right) - (\Phi^{-1}(3/4))^{2}\|\Sigma^{1/2} H^{1/2} e_i\|^{2}\right|
\\
&\le \sum\limits_{i = 1}^d\left|\operatorname{Quant}_{1/2 + \eps}\left(\langle e_i, H^{1/2}Y_1 \rangle^2, \ldots,\langle e_i, H^{1/2}Y_N \rangle^2\right) - (\Phi^{-1}(3/4))^{2}\|\Sigma^{1/2} H^{1/2} e_i\|^{2}\right|
\\
&\qquad+ \sum\limits_{i = 1}^d\left|\operatorname{Quant}_{1/2 - \eps}\left(\langle e_i, H^{1/2}Y_1 \rangle^2, \ldots,\langle e_i, H^{1/2}Y_N \rangle^2\right) - (\Phi^{-1}(3/4))^{2}\|\Sigma^{1/2} H^{1/2} e_i\|^{2}\right|.
\end{align*}
We only consider the first sum, as the second sum is analyzed similarly. Observe that by the spherical symmetry the random variable 
$
\langle e_i, H^{1/2}Y \rangle^2/\|\Sigma^{1/2} H^{1/2} e_i\|^{2}
$ is distributed according to the $\chi^2_1$ distribution. Denote 
\[
S_{N, i} = \{\langle e_i, H^{1/2}Y_1 \rangle^2/\|\Sigma^{1/2} H^{1/2} e_i\|^{2}, \ldots ,\langle e_i, H^{1/2}Y_N \rangle^2/\|\Sigma^{1/2} H^{1/2} e_i\|^{2}\}.
\]
Using the notation from the previous display, triangle inequality and the fact that $F^{-1}_{\chi^2_1}(1/2) = (\Phi^{-1}(3/4))^{2}$, we arrive at 
\begin{align*}
&\left|\operatorname{Quant}_{1/2 + \eps}\left(\langle e_i, H^{1/2}Y_1 \rangle^2, \ldots,\langle e_i, H^{1/2}Y_N \rangle^2\right) - (\Phi^{-1}(3/4))^{2}\|\Sigma^{1/2} H^{1/2} e_i\|^{2}\right|
\\
&\le\|\Sigma^{1/2} H^{1/2} e_i\|^{2}\left|\operatorname{Quant}_{1/2 + \eps}\left(S_{N, i}\right) - \E \operatorname{Quant}_{1/2 + \eps}\left(S_{N, i}\right)\right|
\\
&\qquad+\|\Sigma^{1/2} H^{1/2} e_i\|^{2}\left|\E \operatorname{Quant}_{1/2 + \eps}\left(S_{N, i}\right) - F^{-1}_{\chi^2_1}(1/2 + \eps)\right|
\\
&\qquad+\|\Sigma^{1/2} H^{1/2} e_i\|^{2}\left|F^{-1}_{\chi^2_1}(1/2 + \eps) - F^{-1}_{\chi^2_1}(1/2)\right|
\\
&= (I)_i + (II)_i + (III)_i.
\end{align*}
By Lemma \ref{lem:orderstatofchisq} we have for some $c_1 > 0$,
\[
\left\|(I)_i + (II)_i\right\|_{\psi_1} \le \frac{c_1\|\Sigma^{1/2} H^{1/2} e_i\|^{2}}{\sqrt{N}},\quad \textrm{and therefore,}\quad \left\|\sum\limits_{i = 1}^d\bigl((I)_i + (II)_i\bigr)\right\|_{\psi_1} \le \frac{c_1\tr(\Sigma H)}{\sqrt{N}},
\]
where the last expression follows from the triangle inequality. 
Using the exact form of the inverse cumulative distribution function of the $\chi^2_1$ distribution and the same technique used to bound the difference of quantiles of half-normal distribution, one can verify that for any $\eps \le 1/4$ we have $\left|F^{-1}_{\chi^2_1}(1/2 \pm \eps) - F^{-1}_{\chi^2_1}(1/2)\right| \le c_2\eps$, where $c_2>0$ is an absolute constant. This readily yields
\[
(III)_i = \|\Sigma^{1/2} H^{1/2} e_i\|^{2}\left|F^{-1}_{\chi^2_1}(1/2 + \eps) - F^{-1}_{\chi^2_1}(1/2)\right| \le c_2\|\Sigma^{1/2} H^{1/2} e_i\|^{2}\eps.
\]
Therefore, for $\eps \le 1/4$, we have $\sum\nolimits_{i = 1}^d(III)_i \le c_2\eps\tr(\Sigma H)$.
Combing the above computations and using the tail bound of \cite[Proposition 2.7.1]{Vershynin2016HDP}, we prove the inequality \eqref{eq:firstineq}.

To prove the second part of the bound we propose a slightly different analysis for the term $(I)_i$. Denote
\[
S^{\prime}_{N, i} = \{|\langle e_i, H^{1/2}Y_1 \rangle|/\|\Sigma^{1/2} H^{1/2} e_i\|, \ldots ,|\langle e_i, H^{1/2}Y_N \rangle|/\|\Sigma^{1/2} H^{1/2} e_i\|\},
\]
and observe that $S^{\prime}_{N, i}$ consists of independent half-normal random variables. 
We have
\begin{align*}
(I)_i &= \|\Sigma^{1/2} H^{1/2} e_i\|^{2}\left|\operatorname{Quant}_{1/2 + \eps}\left(S^{\prime}_{N, i}\right) -\sqrt{\E \operatorname{Quant}_{1/2 + \eps}\left(S_{N, i}\right)}\right|
\\
&\qquad\times\left|\operatorname{Quant}_{1/2 + \eps}\left(S^{\prime}_{N, i}\right) +\sqrt{\E \operatorname{Quant}_{1/2 + \eps}\left(S_{N, i}\right)}\right|.
\end{align*}
We first bound the second multiplier of the expression from the last display. By Lemma \ref{lem:orderstatsofhalfnormal}, with probability at least $1-\delta$, we have 
\begin{align}
    \operatorname{Quant}_{1/2 + \eps}(S^{\prime}_{N, i}) + \sqrt{\E \operatorname{Quant}_{1/2 + \eps}\left(S_{N, i}\right)} &\le 2\Phi_{\operatorname{H}}^{-1}(1/2 + \eps) + c_3\sqrt{\frac{\log(1/\delta)}{N}}.
\end{align}
Now observe that the last expression from the previous line is bounded by some absolute constant given that $\eps\in[0, 1/4]$ and $N \ge c_4\log(1/\delta)$. Using Lemma \ref{lem:orderstatsofhalfnormal} once again together with union bound, we bound the term $(I)_i$, with probability at least $1-\delta$, as follows
\begin{align}
    (I)_i \le c_4\|\Sigma^{1/2} H^{1/2} e_i\|^{2}\sqrt{\frac{\log(1/\delta)}{N}}.
\end{align}
By the union bound, for all $i \in [d]$ we have, with probability at least $1-\delta$,
\begin{align}
    (I)_i \le c_3\|\Sigma^{1/2} H^{1/2} e_i\|^{2}\sqrt{\frac{\log d + \log(1/\delta)}{N}},
\end{align}
whenever $N \ge c_2(\log d + \log(1/\delta))$. Taking the sum over all $i \in [d]$ concludes the proof. 
\end{proof}

\section{Concluding remarks}
\label{sec:conlcudingremarks}
Several natural questions follow.
The first is on the existence of computationally efficient estimators achieving our bounds. It is known that getting a polynomial time algorithm with a linear dependence on $\eps$ in the strong contamination model matching the bound of Theorem \ref{thm:maintheorem} is a challenging problem, even when the covariance matrix is identity. Covariance estimation is an even harder problem from the computational perspective. To the best of our knowledge, it is unknown if there is a polynomial time algorithm achieving the statistical performance of Theorem \ref{thm:covarianceestimation} even with the much weaker $\sqrt{\eps}$-dependence on the contamination level.

One simpler question is if our bounds can be generalized beyond the Gaussian case. The answer is yes, and we opted for explicit Gaussian computations only to make our proofs more reader-friendly. In particular, the proof of Theorem \ref{thm:maintheorem} only uses the following properties of the  distribution:
\begin{enumerate}
\item The distribution of $X - \mu$ is symmetric around the origin.
\item The distribution is spherically symmetric. That is, for any $v \in S^{d-1}$, the distribution of $\langle X - \mu, v\rangle/\sqrt{v^{\top}\Sigma v}$ does not depend on $v$. Denote the density function of this distribution by $f$.
\item The inverse of the cumulative distribution function $F$ that corresponds to the density $f$ satisfies $|F^{-1}(1/2 \pm \eps) - F^{-1}(1/2)| \le c\eps$ for some $c > 0$ and small enough $\eps$.
\item The density function $f$ is separated from zero by some absolute constant for all $x \in [F^{-1}(1/2 - \eps), F^{-1}(1/2 + \eps)]$.
\item The distribution corresponding to the density function $f$ is sub-Gaussian. That is, for $Y$ distributed according to this distribution we have $\|Y\|_{\psi_2} \le c$ for some $c > 0$.  
\end{enumerate}
Following the lines of our proof almost verbatim, one can analyze these more general distributions.
It will be interesting to understand if the sub-Gaussian tails assumption (Property 5) can be avoided. In fact, assuming Properties $1$-$4$, and additionally that $\Sigma = I_d$, combining our techniques and the analysis in \cite[Proposition 1.3]{diakonikolas2019recent}, one can build an estimator $\widehat{\mu}$ satisfying, with probability at least $1-\delta$,
\[
\|\widehat{\mu} - \mu\| \le c\left(\sqrt{\frac{d + \log(1/\delta)}{N}} + \eps\right),
\]
whenever $N \ge c_1(d + \log(1/\delta))$. Here $c, c_1 > 0$ are some absolute constants. A similar bound without the sub-Gaussian assumption is also given by Chen, Gao, and Ren \cite[Section 4]{chen2018robust}. In our case, the sub-Gaussian assumption (Property 5) is needed to control the moment generating function when applying Lemma \ref{lem:pacbayes}, while the proof 
in \cite[Proposition 1.3]{diakonikolas2019recent} is based on the union bound over the $\varepsilon$-net for which we do not need sub-Gaussian tails in the \say{large deviation} regime.

Finally, some of the parameters of Theorem \ref{thm:covarianceestimation} are rather hard to estimate without making additional assumptions on the sample size, and confidence level. One can possibly adapt other approaches, such as, for example, Lepskii's method \cite{lepskiui1990problem}. This could provide an alternative way of tuning these parameters.

\paragraph{Acknowledgments.} The authors would like to thank Ankit Pensia for a discussion on differences between contamination models, Vladimir Ulyanov for a discussion on asymptotic laws for sample quantiles, and Arnak Dalalyan for many insightful discussions and useful comments. The work of AM was supported by the grant Investissements d'Avenir (ANR-11-IDEX-0003/Labex Ecodec/ANR-11-LABX-0047) and by the FAST Advance grant.
{\footnotesize
\bibliography{mybib}

\newcommand{\etalchar}[1]{$^{#1}$}
\begin{thebibliography}{KMR{\etalchar{+}}19}

\bibitem[ABM19]{altschuler2019best}
Jason Altschuler, Victor-Emmanuel Brunel, and Alan Malek.
\newblock Best arm identification for contaminated bandits.
\newblock {\em Journal of Machine Learning Research}, 20(91):1--39, 2019.

\bibitem[AC11]{audibert2011robust}
Jean-Yves Audibert and Olivier Catoni.
\newblock Robust linear least squares regression.
\newblock {\em The Annals of Statistics}, 39(5):2766--2794, 2011.

\bibitem[AZ22]{abdalla2022covariance}
Pedro Abdalla and Nikita Zhivotovskiy.
\newblock Covariance estimation: Optimal dimension-free guarantees for
  adversarial corruption and heavy tails.
\newblock {\em arXiv preprint arXiv:2205.08494}, 2022.

\bibitem[Bah66]{bahadur1966note}
Raj Bahadur.
\newblock A note on quantiles in large samples.
\newblock {\em The Annals of Mathematical Statistics}, 37(3):577--580, 1966.

\bibitem[BL19]{bobkov2019one}
Sergey Bobkov and Michel Ledoux.
\newblock {\em One-dimensional empirical measures, order statistics, and
  {Kantorovich} transport distances}, volume 1259 of {\em Memoirs of the
  American Mathematical Society}.
\newblock American Mathematical Society, 2019.

\bibitem[BLM13]{boucheron2013concentration}
St{\'e}phane Boucheron, G{\'a}bor Lugosi, and Pascal Massart.
\newblock {\em Concentration Inequalities: A Nonasymptotic Theory of
  Independence}.
\newblock Oxford University Press, 2013.

\bibitem[BMD22]{Bateni2022NearlyMR}
Amir-Hossein Bateni, Arshak Minasyan, and Arnak~S. Dalalyan.
\newblock Nearly minimax robust estimator of the mean vector by iterative
  spectral dimension reduction.
\newblock {\em arXiv preprint arXiv:2204.02323}, 2022.

\bibitem[Bor75]{borell1975brunn}
Christer Borell.
\newblock The {B}runn-{M}inkowski inequality in {G}auss space.
\newblock {\em Inventiones Mathematicae}, 30(2):207--216, 1975.

\bibitem[BT12]{boucheron2012concentration}
St{\'e}phane Boucheron and Maud Thomas.
\newblock Concentration inequalities for order statistics.
\newblock {\em Electronic Communications in Probability}, 17:1--12, 2012.

\bibitem[Bur97]{burnashev1997asymptotic}
Marat Burnashev.
\newblock Asymptotic expansions for median estimate of a parameter.
\newblock {\em Theory of Probability \& Its Applications}, 41(4):632--645,
  1997.

\bibitem[Cat16]{catoni2016pac}
Olivier Catoni.
\newblock {PAC-Bayesian} bounds for the {Gram} matrix and least squares
  regression with a random design.
\newblock {\em arXiv preprint arXiv:1603.05229}, 2016.

\bibitem[CDG19]{cheng2019high}
Yu~Cheng, Ilias Diakonikolas, and Rong Ge.
\newblock High-dimensional robust mean estimation in nearly-linear time.
\newblock In {\em Proceedings of the thirtieth annual ACM-SIAM symposium on
  discrete algorithms}, pages 2755--2771. SIAM, 2019.

\bibitem[CG17]{catoni2017dimension}
Olivier Catoni and Ilaria Giulini.
\newblock Dimension-free {PAC-Bayesian} bounds for matrices, vectors, and
  linear least squares regression.
\newblock {\em arXiv preprint arXiv:1712.02747}, 2017.

\bibitem[CGR18]{chen2018robust}
Mengjie Chen, Chao Gao, and Zhao Ren.
\newblock Robust covariance and scatter matrix estimation under huber’s
  contamination model.
\newblock {\em The Annals of Statistics}, 46(5):1932--1960, 2018.

\bibitem[DG92]{donoho1992breakdown}
David~L Donoho and Miriam Gasko.
\newblock Breakdown properties of location estimates based on halfspace depth
  and projected outlyingness.
\newblock {\em The Annals of Statistics}, pages 1803--1827, 1992.

\bibitem[DHF07]{de2007extreme}
Laurens De~Haan and Ana Ferreira.
\newblock {\em Extreme Value Theory: An Introduction}.
\newblock Springer Science \& Business Media, 2007.

\bibitem[DK19]{diakonikolas2019recent}
Ilias Diakonikolas and Daniel~M Kane.
\newblock Recent advances in algorithmic high-dimensional robust statistics.
\newblock {\em arXiv preprint arXiv:1911.05911}, 2019.

\bibitem[DKK{\etalchar{+}}17]{DiakonikolasKK017}
Ilias Diakonikolas, Gautam Kamath, Daniel~M. Kane, Jerry Li, Ankur Moitra, and
  Alistair Stewart.
\newblock Being robust (in high dimensions) can be practical.
\newblock In {\em Proceedings of the 34th International Conference on Machine
  Learning, {ICML} 2017}, volume~70, pages 999--1008, 2017.

\bibitem[DKLP22]{diakonikolas2022outlier}
Ilias Diakonikolas, Daniel~M Kane, Jasper~CH Lee, and Ankit Pensia.
\newblock Outlier-robust sparse mean estimation for heavy-tailed distributions.
\newblock {\em arXiv preprint arXiv:2211.16333}, 2022.

\bibitem[DKP20]{diakonikolas2020outlier}
Ilias Diakonikolas, Daniel~M Kane, and Ankit Pensia.
\newblock Outlier robust mean estimation with subgaussian rates via stability.
\newblock {\em Advances in Neural Information Processing Systems},
  33:1830--1840, 2020.

\bibitem[DKS17]{diakonikolas2017statistical}
Ilias Diakonikolas, Daniel~M Kane, and Alistair Stewart.
\newblock Statistical query lower bounds for robust estimation of
  high-dimensional {G}aussians and {G}aussian mixtures.
\newblock In {\em 2017 IEEE 58th Annual Symposium on Foundations of Computer
  Science (FOCS)}, pages 73--84, 2017.

\bibitem[DL22]{depersin2021robustness}
Jules Depersin and Guillaume Lecu{\'e}.
\newblock On the robustness to adversarial corruption and to heavy-tailed data
  of the {S}tahel–{D}onoho median of means.
\newblock {\em Information and Inference: A Journal of the IMA}, 12 2022.

\bibitem[DM22]{dalalyan2022all}
Arnak~S Dalalyan and Arshak Minasyan.
\newblock All-in-one robust estimator of the {G}aussian mean.
\newblock {\em The Annals of Statistics}, 50(2):1193--1219, 2022.

\bibitem[DN04]{david2004order}
Herbert~A David and Haikady~N Nagaraja.
\newblock {\em Order Statistics}.
\newblock John Wiley \& Sons, 2004.

\bibitem[DV75]{donsker1975asymptotic}
Monroe~D. Donsker and Srinivasa Varadhan.
\newblock Asymptotic evaluation of certain markov process expectations for
  large time, {I}.
\newblock {\em Communications on Pure and Applied Mathematics}, 28(1):1--47,
  1975.

\bibitem[Giu18]{giulini2018robust}
Ilaria Giulini.
\newblock Robust dimension-free {G}ram operator estimates.
\newblock {\em Bernoulli}, 24(4B):3864--3923, 2018.

\bibitem[Han22]{han2022exact}
Qiyang Han.
\newblock Exact spectral norm error of sample covariance.
\newblock {\em arXiv preprint arXiv:2207.13594}, 2022.

\bibitem[HL19]{hopkins2019hard}
Samuel~B Hopkins and Jerry Li.
\newblock How hard is robust mean estimation?
\newblock In {\em Conference on Learning Theory}, pages 1649--1682. PMLR, 2019.

\bibitem[HRRS80]{hampel1980robust}
Frank~R Hampel, Peter~J Rousseeuw, Elvezio~M Ronchetti, and Werner~A Stahel.
\newblock {\em Robust statistics: the approach based on influence functions}.
\newblock Wiley, 1980.

\bibitem[Hub64]{Huber_1964}
Peter~J Huber.
\newblock Robust estimation of a location parameter.
\newblock {\em The Annals of Mathematical Statistics}, 35(1):73--101, 1964.

\bibitem[Hub81]{huber1981robust}
Peter~J Huber.
\newblock Robust statistics.
\newblock {\em Wiley Series in Probability and Mathematical Statistics}, 1981.

\bibitem[Kie67]{kiefer1967bahadur}
Jack Kiefer.
\newblock On {B}ahadur's representation of sample quantiles.
\newblock {\em The Annals of Mathematical Statistics}, 38(5):1323--1342, 1967.

\bibitem[KL17]{koltchinskii2017operators}
Vladimir Koltchinskii and Karim Lounici.
\newblock Concentration inequalities and moment bounds for sample covariance
  operators.
\newblock {\em Bernoulli}, 23(1):110--133, 2017.

\bibitem[KMR{\etalchar{+}}19]{ke2019user}
Yuan Ke, Stanislav Minsker, Zhao Ren, Qiang Sun, and Wen-Xin Zhou.
\newblock User-friendly covariance estimation for heavy-tailed distributions.
\newblock {\em Statistical Science}, 34(3):454--471, 2019.

\bibitem[Kol31]{kolmogorov31}
Andrey Kolmogorov.
\newblock La m{\'e}thode de la mediane dans la th{\'e}orie des erreurs.
\newblock {\em Rec. Math. Moscou}, 38(3-4):47--50, 1931.

\bibitem[KZ20]{klochkov2020uniform}
Yegor Klochkov and Nikita Zhivotovskiy.
\newblock Uniform {H}anson-{W}right type concentration inequalities for
  unbounded entries via the entropy method.
\newblock {\em Electronic Journal of Probability}, 25, 2020.

\bibitem[Lep90]{lepskiui1990problem}
Oleg~V Lepskii.
\newblock A problem of adaptive estimation in gaussian white noise.
\newblock {\em Teoriya Veroyatnostei i ee Primeneniya}, 35(3):459--470, 1990.

\bibitem[LM00]{laurent2000adaptive}
Beatrice Laurent and Pascal Massart.
\newblock Adaptive estimation of a quadratic functional by model selection.
\newblock {\em Annals of Statistics}, pages 1302--1338, 2000.

\bibitem[LM19a]{lugosi2019mean}
G{\'a}bor Lugosi and Shahar Mendelson.
\newblock Mean estimation and regression under heavy-tailed distributions: A
  survey.
\newblock {\em Foundations of Computational Mathematics}, 19(5):1145--1190,
  2019.

\bibitem[LM19b]{lugosi2019near}
G{\'a}bor Lugosi and Shahar Mendelson.
\newblock Near-optimal mean estimators with respect to general norms.
\newblock {\em Probability theory and related fields}, 175(3):957--973, 2019.

\bibitem[LM21]{lugosi2021robust}
G{\'a}bor Lugosi and Shahar Mendelson.
\newblock Robust multivariate mean estimation: The optimality of trimmed mean.
\newblock {\em The Annals of Statistics}, 49(1):393--410, 2021.

\bibitem[LRV16]{LaiRV16}
Kevin~A. Lai, Anup~B. Rao, and Santosh~S. Vempala.
\newblock Agnostic estimation of mean and covariance.
\newblock In {\em {IEEE} 57th Annual Symposium on Foundations of Computer
  Science, {FOCS} 2016}, pages 665--674, 2016.

\bibitem[MW22]{minsker2022robust}
Stanislav Minsker and Lang Wang.
\newblock Robust estimation of covariance matrices: Adversarial contamination
  and beyond.
\newblock {\em arXiv preprint arXiv:2203.02880}, 2022.

\bibitem[MZ20]{mendelson2020robust}
Shahar Mendelson and Nikita Zhivotovskiy.
\newblock Robust covariance estimation under ${L}_4-{L}_2$ norm equivalence.
\newblock {\em The Annals of Statistics}, 48(3):1648--1664, 2020.

\bibitem[Oli16]{oliveira2016lower}
Roberto~I Oliveira.
\newblock The lower tail of random quadratic forms with applications to
  ordinary least squares.
\newblock {\em Probability Theory and Related Fields}, 166(3-4):1175--1194,
  2016.

\bibitem[OR22]{oliveira2022improved}
Roberto~I Oliveira and Zoraida~F Rico.
\newblock Improved covariance estimation: Optimal robustness and sub-{G}aussian
  guarantees under heavy tails.
\newblock {\em arXiv preprint arXiv:2209.13485}, 2022.

\bibitem[RL05]{rousseeuw1987robust}
Peter~J Rousseeuw and Annick~M Leroy.
\newblock {\em Robust Regression and Outlier Detection}, volume 589.
\newblock John Wiley \& Sons, 2005.

\bibitem[Sha03]{shao2003mathematical}
Jun Shao.
\newblock {\em Mathematical Statistics}.
\newblock Springer Science \& Business Media, 2003.

\bibitem[Tal14]{Talagrand2014}
Michel Talagrand.
\newblock {\em Upper and Lower Bounds for Stochastic Processes: Modern Methods
  and Classical Problems}, volume~60.
\newblock Springer Science \& Business Media, 2014.

\bibitem[TIS76]{cirel1976norms}
Boris~S Tsirelson, Ildar~A Ibragimov, and Vladimir~N Sudakov.
\newblock Norms of gaussian sample functions.
\newblock In {\em Proceedings of the Third Japan—USSR Symposium on
  Probability Theory}, pages 20--41. Springer, 1976.

\bibitem[Ver16]{Vershynin2016HDP}
Roman Vershynin.
\newblock {\em High-Dimensional Probability: An Introduction with
  Applications}, volume~47 of {\em Cambridge Series in Statistical and
  Probabilistic Mathematics}.
\newblock Cambridge University Press, 2016.

\bibitem[Xia19]{xia2019non}
Dong Xia.
\newblock Non-asymptotic bounds for percentiles of independent non-identical
  random variables.
\newblock {\em Statistics \& Probability Letters}, 152:111--120, 2019.

\bibitem[Zhi21]{zhivotovskiy2021dimension}
Nikita Zhivotovskiy.
\newblock Dimension-free bounds for sums of independent matrices and simple
  tensors via the variational principle.
\newblock {\em arXiv preprint arXiv:2108.08198}, 2021.

\end{thebibliography}
\bibliographystyle{alpha}
}
\end{document}